\documentclass{amsart}

\usepackage{geometry,graphicx,amssymb,amsmath,amsbsy,eucal,amsfonts,mathrsfs,amscd,bm,tcolorbox, subcaption, enumitem, tikz-cd}
\usepackage[hidelinks]{hyperref}
\usepackage[all]{xy}

\numberwithin{equation}{section}
\allowdisplaybreaks[1]

\newtheorem{theorem}{Theorem}[section]
\newtheorem{lemma}[theorem]{Lemma}

\newtheorem{corollary}[theorem]{Corollary}

\theoremstyle{definition}

\theoremstyle{remark}

\newcommand{\R}{\mathbb{R}}
\newcommand{\rot}{\operatorname{rot}}

\newcommand{\dd}{\mathsf{d}}

\renewcommand\div{\operatorname{div}}
\newcommand\sym{\operatorname{sym}}

\newcommand{\curl}{\operatorname{curl}}
\newcommand{\grad}{\operatorname{grad}}
\newcommand{\dive}{{\ensuremath\mathop{\mathrm{div}\,}}}

\newcommand\vskw{\operatorname{vskw}}
\newcommand\mskw{\operatorname{mskw}}
\newcommand\skw{\operatorname{skw}}

\newcommand{\tr}{\mathsf{t}\mathsf{r}}

\newcommand{\inc}{\operatorname{inc}}

\newcommand{\dx}{\,{\rm d}x}
\newcommand{\defm}{\operatorname{def}}
\newcommand{\hess}{\operatorname{hess}}

\newcommand{\sign}{\operatorname{sign}}

\begin{document}
\title[Low-regularity elasticity complexes]{Low-regularity finite element elasticity complexes with hybridizable stresses on tetrahedral Alfeld splits}
\author{Johnny Guzm\'{a}n}%
 \address{Division of Applied Mathematics, Brown University, Box F, 182 George Street, Providence, RI 02912, USA}%
 \email{johnny\_guzman@brown.edu}%
 \author{Xuehai Huang}%
 \address{School of Mathematics, Shanghai University of Finance and Economics, Shanghai 200433, China}%
 \email{huang.xuehai@sufe.edu.cn}%

 \makeatletter
\@namedef{subjclassname@2020}{\textup{2020} Mathematics Subject Classification}
\makeatother

\begin{abstract}
Finite element elasticity complexes of low regularity are constructed on tetrahedral Alfeld splits.  In comparison with existing three-dimensional elasticity complexes on such splits, the complexes constructed here lower both the Sobolev regularity and the polynomial degrees, while ending in a hybridizable $H(\div;\mathbb S)$-conforming symmetric stress space with no vertex degrees of freedom.  The construction is obtained from local Bernstein--Gelfand--Gelfand arguments applied to polynomial de Rham complexes on the Alfeld split.  Two local polynomial elasticity complexes are proved: an $H^2$--$H^1(\inc)$ complex and a lower-regularity $H^1(\curl)$--$H(\inc^+)$ complex.  Their bubble subcomplexes and dimension formulas are derived.  These local exact sequences lead to unisolvent finite elements for the displacement and incompatibility spaces and to global finite element subcomplexes of the corresponding elasticity sequences.  In the lowest-order $H^1(\curl)$--$H(\inc^+)$ finite element complex, the $H(\inc^+;\mathbb S)$-conforming tensor space is piecewise cubic.  At the same order, the terminal stress--displacement pair recovers the Johnson--Mercier--K\v{r}\'{i}\v{z}ek element, while the construction covers higher-order hybridizable symmetric stresses for all $k\ge1$.  A second family gives a low-regularity $H^1$--$H(\inc)$ finite element complex for the standard elasticity sequence for all $k\ge2$. Commuting interpolation diagrams are established for both global complexes.
\end{abstract}

\maketitle


\section{Introduction}

Finite element complexes provide a structural framework for constructing
conforming finite element spaces whose unknowns are linked by differential
operators.  For linear elasticity in three space dimensions, the relevant
continuous complex is
\begin{equation}\label{elascomplex3d}
{\rm RM}
\xrightarrow{\subset}
H^1(\Omega;\mathbb R^3)
\xrightarrow{\defm}
H(\inc,\Omega;\mathbb S)
\xrightarrow{\inc}
H(\div,\Omega;\mathbb S)
\xrightarrow{\div}
L^2(\Omega;\mathbb R^3)
\to0,
\end{equation}
where ${\rm RM}$ is the space of infinitesimal rigid motions,
$\defm=\sym\grad$ is the linearized strain, and $\inc$ is the incompatibility
operator.  The tensor-valued Sobolev spaces in \eqref{elascomplex3d} are
\begin{align*}
H(\inc,\Omega;\mathbb S)
&:=\{\boldsymbol\tau\in L^2(\Omega;\mathbb S):
\inc\boldsymbol\tau\in L^2(\Omega;\mathbb S)\}, \\
H(\div,\Omega;\mathbb S)
&:=\{\boldsymbol\tau\in L^2(\Omega;\mathbb S):
\div\boldsymbol\tau\in L^2(\Omega;\mathbb R^3)\}.
\end{align*}
Thus $H(\inc,\Omega;\mathbb S)$ is the space for symmetric tensor fields with
square-integrable incompatibility, while $H(\div,\Omega;\mathbb S)$ is the
natural space for symmetric stress tensors.  The complex is the linear
elasticity analogue of the de Rham complex and plays an important role in mixed
elasticity and structure-preserving discretizations
\cite{Arnold2006a,ArnoldFalkWinther2006}, intrinsic elasticity and
Saint-Venant compatibility conditions
\cite{geymonat2005some,ciarlet2009intrinsic}, and models of defects and
incompatibility \cite{seeger1961recent,amstutz2018incompatibility}.  It also
gives explicit descriptions of kernels and ranges, which are useful in
stability analysis, preconditioning, and the construction of commuting
projections; see, for example,
\cite{ArnoldFalkWinther2006,ChristiansenGopalakrishnanGuzmanHu2024,ChenHuang2026}.

Constructing finite element subcomplexes of \eqref{elascomplex3d} is delicate
for two related reasons.  First, the stress space must enforce both symmetry and
$H(\div)$ conformity.  Classical polynomial symmetric stress elements are
stable, but they typically involve vertex degrees of freedom and relatively
high polynomial degrees.  Second, the preceding $H(\inc)$ space has
nonstandard traces: tangential--tangential components and second-order surface
differential information enter the Green identity for the incompatibility
operator.  Consequently, conforming $H(\inc)$ elements are substantially more
constrained than standard $H(\curl)$- or $H(\div)$-conforming elements.

The literature contains both constructions of stable symmetric stress spaces
and constructions of full finite element elasticity complexes.  On simplicial meshes, the two-dimensional Arnold--Winther element
\cite{ArnoldWinther2002}, together with its interpretation through finite
element exterior calculus and the Bernstein--Gelfand--Gelfand (BGG)
construction \cite{ArnoldFalkWinther2006}, gives a conforming discretization
of the elasticity complex on triangular meshes.  More systematic
two-dimensional BGG constructions, including elasticity and divdiv complexes
with several smoothness levels, were developed in \cite{ChenHuang2025a}.  In
three dimensions, stable conforming symmetric stress elements on tetrahedral
meshes were constructed in
\cite{ArnoldAwanouWinther2008,HuZhang2015,Hu2015,ChenHuang2022b,HuangZhangZhouZhu2024}.
A full finite element elasticity complex on tetrahedral meshes, involving an
$H(\inc)$-conforming tensor element, was given in \cite{ChenHuang2022}.
Systematic three-dimensional BGG constructions, which derive finite element
complexes from existing complexes and include the elasticity complex as a
central example, were developed in \cite{ChenHuang2026}.

Alfeld-type macroelements provide another route to exact elasticity sequences.
In two dimensions, the Clough--Tocher split is the two-dimensional Alfeld
split.  It underlies the Johnson--Mercier and Arnold--Douglas--Gupta stress
elements \cite{JohnsonMercier1978,ArnoldDouglasGupta1984}, and also appears in
the finite element system approach to elasticity and curvature
\cite{ChristiansenHu2023}.  In three dimensions, a complete discrete elasticity
complex on tetrahedral Alfeld splits was constructed in
\cite{ChristiansenGopalakrishnanGuzmanHu2024}, using smooth finite element de
Rham complexes on Alfeld refinements \cite{FuGuzmanNeilan2020}.  This is the
closest predecessor of the present work and provides the first
three-dimensional Alfeld elasticity complex.  Its construction, however, is
tied to smoother Alfeld de Rham spaces and to a stress space whose degrees of
freedom include vertex data.  Related three-dimensional exact elasticity
sequences on the Worsey--Farin split were developed in
\cite{GongGopalakrishnanGuzmanNeilan2023}.

The Alfeld-split setting is retained here, but the terminal stress spaces are
replaced by the hybridizable $H(\div;\mathbb S)$-conforming symmetric stress
spaces on the Alfeld split introduced in \cite{ChenHuang2025}, for all
polynomial orders $k\ge1$.  At the lowest order, the terminal stress space
together with the corresponding discontinuous displacement space recovers the
Johnson--Mercier--K\v{r}\'{i}\v{z}ek stress--displacement pair
\cite{JohnsonMercier1978,Krizek1982,GopalakrishnanGuzmanLee2025}.  These
stress spaces have degrees of freedom only on faces and in element interiors.
After static condensation, the globally coupled stress variables are therefore
facet variables rather than vertex stress data.  Relative to
\cite{ChristiansenGopalakrishnanGuzmanHu2024}, this replacement changes not
only the stress space in the sequence but also the regularity and polynomial
degree required of the preceding spaces.  The displacement and incompatibility
spaces must therefore be chosen so that the image of $\inc$ is exactly the
divergence-free subspace of the hybridizable symmetric stress space and so
that commuting interpolants can be defined.

At the Sobolev level, two local models are used on each tetrahedron $T$.  The
first is close to the standard elasticity sequence and is used to construct the
global $H^1$--$H(\inc)$ family.  The second lowers the displacement regularity
and leads to the global $H^1(\curl)$--$H(\inc^+)$ family.  The smoother local
sequence is
\begin{equation}\label{elascomplex:localH2H1inc}
{\rm RM}
\xrightarrow{\subset}
H^2(T;\mathbb R^3)
\xrightarrow{\defm}
H^1(\inc,T;\mathbb S)
\xrightarrow{\inc}
H(\div,T;\mathbb S)
\xrightarrow{\div}
L^2(T;\mathbb R^3)
\to0,
\end{equation}
where $H^1(\inc,T;\mathbb S)=H^1(T;\mathbb S)\cap
H(\inc,T;\mathbb S)$.  The lower-regularity local sequence is
\begin{equation}\label{elascomplex:localH1curlinc+}
{\rm RM}
\xrightarrow{\subset}
H^1(\curl,T)
\xrightarrow{\defm}
H(\inc^+,T;\mathbb S)
\xrightarrow{\inc}
H(\div,T;\mathbb S)
\xrightarrow{\div}
L^2(T;\mathbb R^3)
\to0,
\end{equation}
where $H(\inc^+,T;\mathbb S)=H(\inc,T;\mathbb S)\cap
H(\curl,T;\mathbb S)$.  The smoother local sequence gives a global finite
element subcomplex of \eqref{elascomplex3d}, for which the assembled spaces
impose the $H^1$ and $H(\inc)$ traces.  The lower-regularity local sequence
gives a global finite element subcomplex of
\begin{equation}\label{elascomplex:H1curlinc+}
{\rm RM}
\xrightarrow{\subset}
H^1(\curl,\Omega)
\xrightarrow{\defm}
H(\inc^+,\Omega;\mathbb S)
\xrightarrow{\inc}
H(\div,\Omega;\mathbb S)
\xrightarrow{\div}
L^2(\Omega;\mathbb R^3)
\to0,
\end{equation}
where $H(\inc^+,\Omega;\mathbb S)=H(\inc,\Omega;\mathbb S)\cap
H(\curl,\Omega;\mathbb S)$.  Thus both the $H^1$--$H(\inc)$ elasticity complex
and the lower-regularity $H^1(\curl)$--$H(\inc^+)$ elasticity complex are
discretized.

The main contributions are as follows.  First, using the BGG mechanism
\cite{ArnoldHu2021} and the polynomial de Rham complexes on the Alfeld split
\cite{FuGuzmanNeilan2020}, polynomial analogues of the two Sobolev elasticity
sequences above, together with their bubble subcomplexes, are derived; bubble
exactness and dimension formulas are established for both.  Second, finite
element degrees of freedom are given for the $H^1(\curl)$-conforming
displacement space and the $H(\inc^+;\mathbb S)$-conforming tensor space, and
a global finite element subcomplex of \eqref{elascomplex:H1curlinc+} is proved
for all $k\ge1$.  In the lowest-order case, the
$H(\inc^+;\mathbb S)$-conforming tensor space is piecewise cubic, and the
terminal stress--displacement pair recovers the
Johnson--Mercier--K\v{r}\'{i}\v{z}ek element.  Third, a second family is
constructed for the standard $H^1$--$H(\inc)$ elasticity sequence: an
$H(\inc;\mathbb S)$-conforming tensor element is introduced, its commuting
properties are proved, and a global finite element subcomplex of
\eqref{elascomplex3d} is obtained for all $k\ge2$.  Compared with the earlier
three-dimensional Alfeld elasticity complex
\cite{ChristiansenGopalakrishnanGuzmanHu2024}, these constructions lower the
Sobolev regularity of the preceding spaces, reduce the polynomial degree, and
end in a symmetric $H(\div;\mathbb S)$-conforming stress space with no vertex
degrees of freedom; consequently, the stress space is naturally hybridizable.

Several ingredients enter the construction.  The face degrees of freedom for
the $H(\inc)$ spaces are dictated by the Green identity for the incompatibility
operator and by two trace complexes on each triangular face.  These trace
elements are related to two-dimensional $H(\rot\rot)$, $H(\rot)$,
$H(\div\div)$, and $H(\div)$ tensor elements; see
\cite{ChenHuang2025a,ChenHuang2022b,Hu2015,HuMaZhang2021}.  The interior
degrees of freedom are tied to bubble elasticity complexes.  The commuting
projections use modified degrees of freedom for the hybridizable
$H(\div;\mathbb S)$ stress element.  This organization preserves global
exactness while lowering both the regularity and the polynomial degree compared
with the earlier three-dimensional elasticity complex on Alfeld splits in
\cite{ChristiansenGopalakrishnanGuzmanHu2024}.

The paper is organized as follows.  Section~\ref{sec:notation} fixes notation,
algebraic conventions, differential operators, and the trace identities for the
incompatibility operator.  Section~\ref{sec:localcomplex} proves the local
elasticity complexes on a tetrahedral Alfeld split and records the bubble
exactness and dimension formulas.  Section~\ref{sec:femelascomplexinc+}
constructs the finite element complex for the
$H^1(\curl)$--$H(\inc^+)$ sequence, including unisolvence, global exactness,
and commuting interpolation operators.  Section~\ref{sec:femelascomplexinc}
constructs the lower-regularity $H^1$--$H(\inc)$ finite element complex and
proves the corresponding commuting diagram.  The appendix contains the proofs
of the bubble exactness results used in the main text.

\section{Preliminaries}\label{sec:notation}
This section fixes the notation and conventions used throughout the paper.  We
first introduce the function spaces, polynomial spaces, and geometric notation
for tetrahedra and their Alfeld splits.  We then specify the algebraic
conventions for vector--matrix products and define the differential operators
appearing in the elasticity complex.  Finally, we record the trace identities
and Green's formula used in the construction of the local finite element
spaces.

\subsection{Geometric and polynomial notation}
Let $\Omega\subset\mathbb R^3$ be a bounded polyhedral domain with boundary
$\partial\Omega$.  For a subdomain $D\subseteq\Omega$, we use the standard
Sobolev spaces $H^m(D)$ and $H_0^m(D)$ and write $L^2(D):=H^0(D)$.  The
$L^2$ inner product over $D$ is denoted by $(\cdot,\cdot)_D$, and $L_0^2(D)$
denotes the subspace of $L^2(D)$ consisting of functions with vanishing mean.
If $U,V\subseteq L^2(D)$, then
\[
U/V:=\{u\in U:(u,v)_D=0\text{ for all }v\in V\}.
\]
Thus $U/V$ denotes the $L^2(D)$-orthogonal complement of $V$ in $U$, not an
abstract quotient space.

For an integer $k\ge0$, let $\mathbb P_k(D)$ denote the space of polynomials on
$D$ of total degree at most $k$, with the convention
$\mathbb P_k(D)=\{0\}$ for $k<0$.  The outward unit normal to $\partial D$ is
denoted by $\boldsymbol n_{\partial D}$, or simply by $\boldsymbol n$ when the
domain is clear.

Let $\mathcal T_h$ be a tetrahedral mesh of $\Omega$, with mesh size $h$.  If
$T$ is a $d$-simplex, $d=2,3$, then $\Delta(T)$ denotes the set of all
subsimplices of $T$, and $\Delta_\ell(T)$ denotes the set of
$\ell$-dimensional subsimplices, $0\le \ell\le d$.  Thus
$\Delta_0(T)=\{\texttt{v}_0,\ldots,\texttt{v}_d\}$ is the vertex set and
$\Delta_d(T)=\{T\}$.  Similarly, $\Delta_\ell(\mathcal T_h)$ denotes the set
of all $\ell$-dimensional subsimplices of the mesh.  For a tetrahedron $T$ with
vertices $\texttt{v}_0,\ldots,\texttt{v}_3$ and $0\leq i\leq 3$, $F_i$
denotes the face opposite to $\texttt{v}_i$, and $\boldsymbol n_{F_i}$ denotes
the outward unit normal to $F_i$.  We write $\lambda_i$ for the barycentric
coordinate associated with $\texttt{v}_i$, and $\boldsymbol t_{i,j}$ for the
tangent vector from $\texttt{v}_i$ to $\texttt{v}_j$.

Orientations of lower-dimensional subsimplices are fixed once and for all.  For
each edge $e$, choose a unit tangent vector $\boldsymbol t_e$ and two unit
normal vectors $\boldsymbol n_1^e$ and $\boldsymbol n_2^e$.  For each face
$F$, choose a unit normal vector $\boldsymbol n_F$ and two tangential vectors
$\boldsymbol t_1^F$ and $\boldsymbol t_2^F$.  When no confusion can arise,
these vectors are abbreviated by
$\boldsymbol t$, $\boldsymbol n_1$ and $\boldsymbol n_2$, and
$\boldsymbol n$, $\boldsymbol t_1$ and $\boldsymbol t_2$,
respectively.  On a conforming mesh, edge and face orientations are chosen
globally rather than elementwise.  In expressions such as $\partial_n u$, the
roman letter $n$ indicates differentiation in the normal direction.  If
$e\in\Delta_1(F)$, let $\boldsymbol n_{F,e}$ be the unit vector tangent to $F$
and normal to $e$ induced by the orientation of $F$, and set
\[
\boldsymbol t_{F,e}:=\boldsymbol n\times\boldsymbol n_{F,e}.
\]

We will use the Alfeld split of a tetrahedron throughout.  Let $\texttt{v}_c$
be the barycenter of a tetrahedron $T$.  The Alfeld split $T^{\rm R}$ is
obtained by joining $\texttt{v}_c$ to all vertices of $T$.  We denote by $T_i$
the subtetrahedron whose vertices are $\texttt{v}_c$ together with all vertices
of $T$ except $\texttt{v}_i$; thus
\[
T^{\rm R}=\{T_i:0\le i\le3\}.
\]
The corresponding global Alfeld refinement of $\mathcal T_h$ is denoted by
$\mathcal T_h^{\rm R}$.  Given a collection $\mathcal S$ of tetrahedra, let
$\omega_{\mathcal S}:=\bigcup_{T\in\mathcal S}T$ and define the broken
polynomial space
\[
\mathbb P_k^{-1}(\mathcal S)
:=\{q\in L^2(\omega_{\mathcal S}):
       q|_T\in\mathbb P_k(T)\text{ for each }T\in\mathcal S\}.
\]
The superscript $-1$ indicates that no continuity is imposed across interfaces
between elements of $\mathcal S$.  We also set
\begin{equation*}
\mathbb P_k^{\grad}(\mathcal S)
:=H^1(\omega_{\mathcal S})\cap\mathbb P_k^{-1}(\mathcal S)
=\{q\in H^1(\omega_{\mathcal S}):
q|_T\in\mathbb P_k(T)\text{ for each }T\in\mathcal S\}.
\end{equation*}
For a face $F\in\Delta_2(T)$, denote by $b_F$ the cubic face bubble function.

\subsection{Algebraic conventions}
We next specify the conventions for products involving vectors and matrices.
Following \cite{ChenHuang2022}, products between a vector and a matrix are
interpreted according to the side on which the vector appears.  For a vector
$\boldsymbol b$ and a matrix $\boldsymbol A$, the products
\[
\boldsymbol b\cdot\boldsymbol A,
\qquad
\boldsymbol b\times\boldsymbol A
\]
are taken column-wise.  Conversely, the products
\[
\boldsymbol A\cdot\boldsymbol b,
\qquad
\boldsymbol A\times\boldsymbol b
\]
are taken row-wise.  Equivalently, $\boldsymbol b$ is regarded as a column
vector when it acts from the left and as the row vector
$\boldsymbol b^{\intercal}$ when it acts from the right.  This convention is
purely notational and avoids repeated transposes.

Since column-wise and row-wise products act on different indices, the order of
mixed products is unambiguous.  For example,
\[
\boldsymbol b\times\boldsymbol A\times\boldsymbol c
:= (\boldsymbol b\times\boldsymbol A)\times\boldsymbol c
 = \boldsymbol b\times(\boldsymbol A\times\boldsymbol c).
\]
The same convention applies to mixed products such as
$\boldsymbol b\cdot\boldsymbol A\cdot\boldsymbol c$ and
$\boldsymbol b\cdot\boldsymbol A\times\boldsymbol c$, and parentheses will
usually be omitted.  Transposition reverses the order of the factors; moreover,
a cross product changes sign.  Thus
\[
(\boldsymbol b\cdot\boldsymbol A)^{\intercal}
   = \boldsymbol A^{\intercal}\cdot\boldsymbol b,
\qquad
(\boldsymbol b\times\boldsymbol A)^{\intercal}
   =-\boldsymbol A^{\intercal}\times\boldsymbol b.
\]

For column vectors $\boldsymbol u$ and $\boldsymbol v$, their tensor product is
\[
\boldsymbol u\otimes\boldsymbol v
:=\boldsymbol u\boldsymbol v^{\intercal}.
\]
We also write $\boldsymbol u\boldsymbol v$ for the same rank-one matrix.  With
this notation, row-wise and column-wise products with another vector
$\boldsymbol x$ act on the adjacent factor:
\begin{align*}
\boldsymbol x\cdot(\boldsymbol u\boldsymbol v)
   &= (\boldsymbol x\cdot\boldsymbol u)\boldsymbol v^{\intercal},
&
(\boldsymbol u\boldsymbol v)\cdot\boldsymbol x
   &= \boldsymbol u(\boldsymbol v\cdot\boldsymbol x), \\
\boldsymbol x\times(\boldsymbol u\boldsymbol v)
   &= (\boldsymbol x\times\boldsymbol u)\boldsymbol v,
&
(\boldsymbol u\boldsymbol v)\times\boldsymbol x
   &= \boldsymbol u(\boldsymbol v\times\boldsymbol x).
\end{align*}

We denote by $\mathbb M$ the space of all $3\times3$ matrices, by $\mathbb S$
the subspace of symmetric matrices, and by $\mathbb K$ the subspace of
skew-symmetric matrices.  Every $\boldsymbol B\in\mathbb M$ admits the
decomposition
\[
\boldsymbol B
= \sym\boldsymbol B+\skw\boldsymbol B
:= \frac12(\boldsymbol B+\boldsymbol B^{\intercal})
  +\frac12(\boldsymbol B-\boldsymbol B^{\intercal}).
\]
The map $\mskw:\mathbb R^3\to\mathbb K$ is defined by
\[
\mskw \boldsymbol \omega :=
\begin{pmatrix}
0 & -\omega_3 & \omega_2 \\
\omega_3 & 0 & -\omega_1\\
-\omega_2 & \omega_1 & 0
\end{pmatrix},
\qquad
\boldsymbol\omega=(\omega_1,\omega_2,\omega_3)^{\intercal}.
\]
This map is an isomorphism.  We define
\[
\vskw:\mathbb M\to\mathbb R^3,
\qquad
\vskw := \mskw^{-1}\circ \skw.
\]
Finally, for a scalar function space $B(D)$, we use the compact notation
\[
B(D;\mathbb X):=B(D)\otimes\mathbb X,
\qquad
\mathbb X\in\{\mathbb R^d,\mathbb M,\mathbb S,\mathbb K\}.
\]

\subsection{Differential operators and function spaces}
Let $\nabla=(\partial_1,\partial_2,\partial_3)^{\intercal}$.  For a vector field
$\boldsymbol v$, we use
\[
\nabla\boldsymbol v:=\nabla\otimes\boldsymbol v,
\qquad
\grad\boldsymbol v:=(\nabla\boldsymbol v)^{\intercal},
\qquad
\curl\boldsymbol v:=\nabla\times\boldsymbol v,
\qquad
\div\boldsymbol v:=\nabla\cdot\boldsymbol v.
\]
The symmetric gradient, also denoted by $\varepsilon(\boldsymbol v)$ in
elasticity, is
\[
\defm\boldsymbol v
:=\sym\grad\boldsymbol v
=\frac12\big(\grad\boldsymbol v+(\grad\boldsymbol v)^{\intercal}\big).
\]
With the above convention for $\mskw$, the gradient decomposes as
\begin{equation}\label{eq:gradudecomp}
\grad \boldsymbol  v
=\defm\boldsymbol  v+\frac{1}{2}\mskw(\nabla\times\boldsymbol  v).
\end{equation}

For matrix-valued fields, $\curl$ and $\div$ are applied row-wise:
\[
\curl\boldsymbol\tau=(\nabla\times\boldsymbol\tau^{\intercal})^{\intercal},
\qquad
\div\boldsymbol\tau=(\nabla\cdot\boldsymbol\tau^{\intercal})^{\intercal}.
\]
We use the following form of the incompatibility operator:
\[
\inc\boldsymbol\tau:=\curl S^{-1}(\curl\boldsymbol\tau),
\]
where
\[
S\boldsymbol\sigma:=\boldsymbol\sigma^{\intercal}-(\tr\boldsymbol\sigma)\boldsymbol I,
\qquad
S^{-1}\boldsymbol\sigma:=\boldsymbol\sigma^{\intercal}
    -\frac12(\tr\boldsymbol\sigma)\boldsymbol I.
\]
The operator $\inc$ depends only on the symmetric part of its argument:
\[
\inc\boldsymbol\tau=\inc(\sym\boldsymbol\tau).
\]
In particular, if $\boldsymbol\tau$ is symmetric, then
\begin{equation}\label{eq:incandcott}
\inc\boldsymbol\tau
=\curl(\curl\boldsymbol\tau)^{\intercal}.
\end{equation}

We also require surface differential operators.  Let $F$ be a face with unit
normal vector $\boldsymbol n$.  The tangential projection onto the plane of
$F$ is
\[
\Pi_F\boldsymbol v
:=(\boldsymbol n\times\boldsymbol v)\times\boldsymbol n
=\boldsymbol n\times(\boldsymbol v\times\boldsymbol n)
=-\boldsymbol n\times(\boldsymbol n\times\boldsymbol v)
=(\boldsymbol I-\boldsymbol n\boldsymbol n^{\intercal})\boldsymbol v.
\]
Define
\[
\nabla_F:=\Pi_F\nabla,
\qquad
\nabla_F^{\bot}:=-\boldsymbol n\times\nabla.
\]
For a scalar function $v$,
\begin{align*}
\grad_F v=\nabla_F v
   &=\Pi_F(\nabla v)
     =-\boldsymbol n\times(\boldsymbol n\times\nabla v), \\
\curl_F v=\nabla_F^{\bot}v
   &=-\boldsymbol n\times\nabla v
     =-\boldsymbol n\times\nabla_F v.
\end{align*}
For a vector field $\boldsymbol v$, the surface divergence and surface rotation
are
\[
\div_F\boldsymbol v:=\nabla_F\cdot\boldsymbol v
=\nabla_F\cdot(\Pi_F\boldsymbol v),
\]
and
\[
\rot_F\boldsymbol v
:=-\nabla_F^{\bot}\cdot\boldsymbol v
=(\boldsymbol n\times\nabla)\cdot\boldsymbol v
=\boldsymbol n\cdot(\nabla\times\boldsymbol v).
\]
Thus $\rot_F\boldsymbol v$ is the normal component of $\curl\boldsymbol v$.
Define the surface deformation operator by
\[
\defm_F(\boldsymbol v)
:=\Pi_F\defm(\boldsymbol v)\Pi_F= \sym(\grad_F(\Pi_F\boldsymbol v)).
\]

We use the following Sobolev spaces associated with these differential
operators.  For a domain $D\subset\Omega$, define
\begin{align*}
H(\div,D)
&:=\{\boldsymbol v\in L^2(D;\mathbb R^3):\div\boldsymbol v\in L^2(D)\}, \\
H(\curl,D)
&:=\{\boldsymbol v\in L^2(D;\mathbb R^3):
\curl\boldsymbol v\in L^2(D;\mathbb R^3)\}, \\
H^1(\curl,D)
&:=\{\boldsymbol v\in H^1(D;\mathbb R^3):
\curl\boldsymbol v\in H^1(D;\mathbb R^3)\}.
\end{align*}
We denote by $H_0(\div,D)$ and $H_0(\curl,D)$ the subspaces with vanishing
normal and tangential traces on $\partial D$, respectively.  We further define
\begin{equation*}
H_0^1(\curl,D):=\{\boldsymbol{v}\in H^1(\curl,D):
\boldsymbol{v}=\curl\boldsymbol{v}=0 \;\;\textrm{ on } \partial D\}.
\end{equation*}
For $\mathbb X\in\{\mathbb M,\mathbb S\}$, the corresponding matrix-valued
spaces are
\begin{align*}
H(\div,D;\mathbb X)
&:=\{\boldsymbol\tau\in L^2(D;\mathbb X):
\div\boldsymbol\tau\in L^2(D;\mathbb R^3)\}, \\
H(\curl,D;\mathbb X)
&:=\{\boldsymbol\tau\in L^2(D;\mathbb X):
\curl\boldsymbol\tau\in L^2(D;\mathbb M)\}, \\
H(\inc,D;\mathbb S)
&:=\{\boldsymbol\tau\in L^2(D;\mathbb S):
\inc\boldsymbol\tau\in L^2(D;\mathbb S)\}.
\end{align*}
Denote by $H_0(\div,D;\mathbb X)$ the subspace of
$H(\div,D;\mathbb X)$ with vanishing normal trace
$\boldsymbol\tau\boldsymbol n=\boldsymbol 0$ on $\partial D$, and by
$H_0(\curl,D;\mathbb X)$ the subspace of $H(\curl,D;\mathbb X)$ with
vanishing row-wise tangential trace on $\partial D$.
Finally, we set
\begin{align*}   
H^1(\curl,D;\mathbb M)&:=\mathbb R^3\otimes H^1(\curl,D), \\
H(\inc^+,D;\mathbb S)&:=H(\inc,D;\mathbb S)\cap H(\curl,D;\mathbb S), \\
H^1(\inc,D;\mathbb S)&:=H(\inc,D;\mathbb S)\cap H^1(D;\mathbb S).
\end{align*}

\subsection{Trace operators and Green's identities}
We now record the trace operators for $\inc$ used in the construction and
analysis of the local complexes.  For a smooth symmetric tensor
$\boldsymbol\tau$ and a face $F$, define
\[
\tr_1(\boldsymbol\tau):=\Pi_F\boldsymbol\tau\Pi_F,
\qquad
\tr_2(\boldsymbol\tau):=2\defm_F(\boldsymbol n\cdot\boldsymbol\tau\Pi_F)-\Pi_F\partial_n\boldsymbol\tau\Pi_F.
\]
The second trace has the following equivalent form
\cite[Lemma~4.1]{ChenHuang2022}:
\begin{equation}\label{eq:20240319tr2inc}
\tr_2(\boldsymbol\tau)
=\boldsymbol n\times(\curl\boldsymbol\tau)^{\intercal}\Pi_F
  +\grad_F(\Pi_F\boldsymbol\tau\boldsymbol n)
=-\Pi_F(\curl\boldsymbol\tau)\times\boldsymbol n
  +\nabla_F(\boldsymbol n\cdot\boldsymbol\tau\,\Pi_F).
\end{equation}
Moreover, \cite[Lemma~4.6]{ChenHuang2022} and
\cite[Lemma~5.7]{ChristiansenGopalakrishnanGuzmanHu2024} imply
\begin{align}
\boldsymbol n\cdot(\inc\boldsymbol\tau)\cdot\boldsymbol n
   &=\rot_F\rot_F(\tr_1(\boldsymbol\tau)),
\label{eq:inctr1}\\
\boldsymbol n\times(\inc\boldsymbol\tau)\cdot\boldsymbol n
   &=\rot_F\tr_2(\boldsymbol\tau).
\label{eq:inctr2}
\end{align}
For a smooth vector field $\boldsymbol v$, the traces of the symmetric gradient
are \cite[Lemma~4.5]{ChenHuang2022}
\begin{align}\label{eq:trdef}
\tr_1(\defm(\boldsymbol v))
   &=\defm_F(\Pi_F\boldsymbol v),
&
\tr_2(\defm(\boldsymbol v))
   &=\nabla_F^2(\boldsymbol v\cdot\boldsymbol n).
\end{align}

The following edge traces enter the two-dimensional Green identity on each
face.  For a smooth tensor $\boldsymbol\tau$ and an edge $e\subset\partial F$,
define
\begin{align*}
\tr_1^{F}(\boldsymbol\tau)
&:=\boldsymbol t_e^{\intercal}\boldsymbol\tau\boldsymbol t_e,\\
\tr_2^{F}(\boldsymbol\tau)
&:=-\partial_{\boldsymbol t_e}
    (\boldsymbol t_e^{\intercal}\boldsymbol\tau\boldsymbol n_{F,e})
  +\boldsymbol t_{F,e}^{\intercal}\rot_F\boldsymbol\tau,\\
\tr^{F}(\boldsymbol\tau)
&:=\boldsymbol\tau\boldsymbol t_{F,e}.
\end{align*}
On each edge $e\in\Delta_1(F)$, the face trace $\tr_1$ is compatible with these
edge traces:
\begin{equation}\label{eq:trrotFrotFtr1inc}
\tr_1^{F}(\tr_1(\boldsymbol\tau))=\tr_1^{F}(\boldsymbol\tau),
\qquad
\tr_2^{F}(\tr_1(\boldsymbol\tau))=\tr_2^{F}(\boldsymbol\tau).
\end{equation}
Similarly, using \eqref{eq:20240319tr2inc}, one obtains
\begin{equation}\label{eq:trrotFtr2inc}
\tr^{F}(\tr_2(\boldsymbol\tau))
=\boldsymbol n\times(\curl\boldsymbol\tau)^{\intercal}\boldsymbol t_{F,e}
 +\partial_{t_{F,e}}(\Pi_F\boldsymbol\tau\boldsymbol n).
\end{equation}

We conclude this section with the Green identity for the operator
$\rot_F\rot_F$ on a polygonal face.  This identity is the rotated counterpart
of the Green identity for $\div_F\div_F$; see
\cite[Lemma~4.2]{ChenHuang2022a}.

\begin{lemma}\label{lm:Green2D}
Let $F$ be a polygon.  For any
$\boldsymbol\tau\in\mathcal C^2(F;\mathbb S)$ and $v\in H^2(F)$, we have
\begin{align}
(\rot_F\rot_F\boldsymbol\tau,v)_F
&=(\boldsymbol\tau,\curl_F^2v)_F
 + \sum_{e\in\Delta_1(F)}\sum_{\delta\in\partial e}
 \sign_{e,\delta}
 \big(\boldsymbol t_{F,e}^{\intercal}\boldsymbol\tau\boldsymbol n_{F,e}\big)(\delta)
 v(\delta) \notag\\
&\quad
 -\sum_{e\in\Delta_1(F)}
 \left[
 \big(\tr_1^{F}(\boldsymbol\tau),\partial_{n_{F,e}}v\big)_e
 -\big(\tr_2^{F}(\boldsymbol\tau),v\big)_e
 \right].
\label{eq:greenidentityrotrot2D}
\end{align}
Here
\[
\sign_{e,\delta}:=
\begin{cases}
1, & \text{if } \delta \text{ is the endpoint of } e
     \text{ induced by } \boldsymbol t_{F,e},\\
-1, & \text{if } \delta \text{ is the starting point of } e
      \text{ induced by } \boldsymbol t_{F,e}.
\end{cases}
\]
\end{lemma}

\begin{proof}
Applying Green's identity for $\rot_F$ twice gives
\[
(\rot_F\rot_F\boldsymbol\tau,v)_F
=(\boldsymbol\tau,\curl_F^2v)_F
 +(\boldsymbol\tau\boldsymbol t_{F,e},\curl_Fv)_{\partial F}
 +(\boldsymbol t_{F,e}^{\intercal}\rot_F\boldsymbol\tau,v)_{\partial F},
\]
where the boundary terms are understood edge by edge on $\partial F$.  Since
\[
\boldsymbol t_{F,e}\cdot\curl_Fv=-\partial_{n_{F,e}}v,
\qquad
\boldsymbol n_{F,e}\cdot\curl_Fv=\partial_{t_{F,e}}v,
\]
we obtain
\begin{align*}
(\boldsymbol\tau\boldsymbol t_{F,e},\curl_Fv)_{\partial F}
&=(\boldsymbol t_{F,e}^{\intercal}\boldsymbol\tau\boldsymbol t_{F,e},
   \boldsymbol t_{F,e}\cdot\curl_Fv)_{\partial F}
 +(\boldsymbol n_{F,e}^{\intercal}\boldsymbol\tau\boldsymbol t_{F,e},
   \boldsymbol n_{F,e}\cdot\curl_Fv)_{\partial F}\\
&=-(\boldsymbol t_{F,e}^{\intercal}\boldsymbol\tau\boldsymbol t_{F,e},
    \partial_{n_{F,e}}v)_{\partial F}
 +(\boldsymbol n_{F,e}^{\intercal}\boldsymbol\tau\boldsymbol t_{F,e},
    \partial_{t_{F,e}}v)_{\partial F}.
\end{align*}
Consequently,
\begin{align*}
(\rot_F\rot_F\boldsymbol\tau,v)_F
&=(\boldsymbol\tau,\curl_F^2v)_F
 -(\boldsymbol t_{F,e}^{\intercal}\boldsymbol\tau\boldsymbol t_{F,e},
   \partial_{n_{F,e}}v)_{\partial F}\\
&\quad
 +(\boldsymbol n_{F,e}^{\intercal}\boldsymbol\tau\boldsymbol t_{F,e},
   \partial_{t_{F,e}}v)_{\partial F}
 +(\boldsymbol t_{F,e}^{\intercal}\rot_F\boldsymbol\tau,v)_{\partial F}.
\end{align*}
Integrating by parts along each edge $e\in\Delta_1(F)$ and using the
definitions of $\tr_1^F$ and $\tr_2^F$ yields
\eqref{eq:greenidentityrotrot2D}.
\end{proof}

\section{Local complexes on the Alfeld split}\label{sec:localcomplex}
This section establishes the local polynomial elasticity complexes that provide
the local algebraic input for the global finite element complexes constructed
below.
Let $T$ be a tetrahedron and let $T^{\rm R}$ denote its Alfeld split.  For
$k\geq1$, the smoother, $H^2$--$H^1(\inc)$ type, sequence is
\begin{equation}\label{eq:localelascomplex3d}
{\rm RM}
\xrightarrow{\subset}
V_{k+3}^{\hess}(T^{\rm R};\mathbb R^3)
\xrightarrow{\defm}
\Sigma_{k+2}^{1,\inc}(T;\mathbb S)
\xrightarrow{\inc}
\Sigma_k^{\div}(T;\mathbb S)
\xrightarrow{\div}
\mathbb P_{k-1}^{-1}(T^{\rm R};\mathbb R^3)
\to 0,
\end{equation}
whereas the lower-regularity, $H^1(\curl)$--$H(\inc^+)$ type, sequence is
\begin{equation}\label{eq:localelascomplex13d}
{\rm RM}
\xrightarrow{\subset}
V_{k+3}^{1,\curl}(T^{\rm R})
\xrightarrow{\defm}
\Sigma_{k+2}^{\inc^+}(T;\mathbb S)
\xrightarrow{\inc}
\Sigma_k^{\div}(T;\mathbb S)
\xrightarrow{\div}
\mathbb P_{k-1}^{-1}(T^{\rm R};\mathbb R^3)
\to 0.
\end{equation}
The two sequences have the same last two spaces and differ only in the
regularity imposed on the first two nontrivial spaces.  The local spaces are
defined by
\begin{align*}
V_{k+3}^{\hess}(T^{\rm R}) &:=\mathbb P_{k+3}^{-1}(T^{\rm R})\cap H^2(T), \\
V_{k+3}^{1,\curl}(T^{\rm R})
&:= \mathbb P_{k+3}^{-1}(T^{\rm R};\mathbb R^3)
\cap H^1(\curl,T), \\
\Sigma_{k+2}^{1,\inc}(T;\mathbb S)
&:=\{
\boldsymbol\tau\in
\mathbb P_{k+2}^{-1}(T^{\rm R};\mathbb S)
\cap H^1(\inc,T;\mathbb S): \textrm{ $\boldsymbol{\tau}$ is $C^1$-continuous}
 \\
&\qquad\qquad\qquad\qquad\qquad\qquad\qquad\qquad\quad \textrm{at all vertices of $T^{\rm R}$} \}, \\
\Sigma_{k+2}^{\inc^+}(T;\mathbb S)
&:=\{
\boldsymbol\tau\in
\mathbb P_{k+2}^{-1}(T^{\rm R};\mathbb S)
\cap H(\inc^+,T;\mathbb S): \textrm{ $\boldsymbol{\tau}$ is continuous}
 \\
&\qquad\qquad\qquad\qquad\qquad\qquad\qquad\qquad\quad \textrm{at all vertices of $T^{\rm R}$} \}, \\
 \Sigma_k^{\div}(T;\mathbb S)
&:=  \Sigma_k^{\div}(T;\mathbb M)\cap\ker(\vskw)
= \mathbb P_k^{-1}(T^{\rm R};\mathbb S) \cap H(\div, T;\mathbb S),
\end{align*}
where
\[
\Sigma_k^{\div}(T;\mathbb M)
:=\mathbb R^3\otimes V_k^{\div}(T^{\rm R}), \quad
V_k^{\div}(T^{\rm R})
:= \mathbb P_k^{-1}(T^{\rm R};\mathbb R^3) \cap H(\div,T).
\]
Complexes \eqref{eq:localelascomplex3d} and
\eqref{eq:localelascomplex13d} are polynomial discretizations of the local
Sobolev complexes \eqref{elascomplex:localH2H1inc} and
\eqref{elascomplex:localH1curlinc+}, respectively.

The following dimension formulas are used in the finite element construction:
\begin{align}
\label{eq:Vhessdim}
\dim V_{k+3}^{\hess}(T^{\rm R})&= \binom{k+6}{3}+3\binom{k+2}{3}=\frac{2}{3}(k^3+6k^2+20k+30), \\     
\label{eq:Vkp2H1curldim}
\dim V_{k+3}^{1,\curl}(T^{\rm R})&= (k+3)(2k^2+9k+22), \\
\label{eq:Sigmma1incSdim}
\dim\Sigma_{k+2}^{1,\inc}(T;\mathbb S)& = 4k^3+21k^2+53k+60, \\
\label{eq:Sigmmainc+Sdim}
\dim\Sigma_{k+2}^{\inc^+}(T;\mathbb S)& = 4k^3+24k^2+62k+66, \\
\label{eq:SigmmadivSdim}
\dim \Sigma_k^{\div}(T;\mathbb S) &=(4k+3)(k+1)(k+2).
\end{align}
The first two formulas, \eqref{eq:Vhessdim} and
\eqref{eq:Vkp2H1curldim}, are taken from
\cite[p.~1076]{FuGuzmanNeilan2020}.  The dimensions of the two incompatibility
spaces, \eqref{eq:Sigmma1incSdim} and \eqref{eq:Sigmmainc+Sdim}, are derived
in Lemmas~\ref{lm:Sigmakincdim1} and~\ref{lm:Sigmakincdim2} from the exactness
of \eqref{eq:localelascomplex3d} and \eqref{eq:localelascomplex13d},
respectively.  Formula \eqref{eq:SigmmadivSdim} follows from
Corollary~\ref{cor:vskwonto}.

\subsection{Exactness of the smoother local complex}
We first prove the exactness of \eqref{eq:localelascomplex3d}.  The row-wise
de Rham complexes give the following local BGG diagram:
\begin{equation}\label{eq:localbggdiagram_1}
\begin{tikzcd}[column sep=small]
V_{k+3}^{\hess}(T^{\rm R};\mathbb R^3) \arrow{r}{\grad} & \Sigma_{k+2}^{1,\curl}(T;\mathbb M)
 \arrow{r}{\curl}
 &
\mathbb P_{k+1}^{\grad}(T^{\rm R};\mathbb M)
 \arrow{r}{\div}
 & \mathbb P_{k}^{-1}(T^{\rm R};\mathbb R^3) \to 0
 \\
V_{k+2}^{\hess}(T^{\rm R};\mathbb R^3)\arrow[ur,swap,"\mskw"'] \arrow{r}{\grad}  & 
\mathbb P_{k+1}^{\grad}(T^{\rm R};\mathbb M)
 \arrow[ur,swap,"S"'] \arrow{r}{\curl}
 &
\Sigma_k^{\div}(T;\mathbb M)
 \arrow[ur,swap,"-2\vskw"'] \arrow{r}{\div} \arrow[r]
 & \mathbb P_{k-1}^{-1}(T^{\rm R};\mathbb R^3)\to0,
\end{tikzcd}
\end{equation}
where
\[
\Sigma_{k+2}^{1,\curl}(T;\mathbb M)
:=\mathbb R^3\otimes V_{k+2}^{1,\curl}(T^{\rm R}).
\]
Both rows are exact; see \cite{FuGuzmanNeilan2020}.  Moreover, functions in
$\Sigma_{k+2}^{1,\curl}(T;\mathbb M)$ are $C^1$-continuous at all vertices of
$T^{\rm R}$; see \cite[Lemma~4.4]{FuGuzmanNeilan2020}.  The diagonal arrows
encode the algebraic identities
\cite{ChristiansenGopalakrishnanGuzmanHu2024,ChenHuang2026}
\begin{align}
\notag
\div(S\boldsymbol{\tau})&=2\vskw(\curl\boldsymbol{\tau}),
\qquad
\forall\,\boldsymbol\tau\in H(\curl,T;\mathbb M), \\
\label{eq:anticommutprop2.1}
S \grad\boldsymbol{v}&=-\curl (\mskw\boldsymbol{v}),
\quad\;\;
\forall\,\boldsymbol{v} \in H^1(T;\R^3),
\end{align}
so the diagram \eqref{eq:localbggdiagram_1} is anticommutative.  It is the
polynomial analogue of the corresponding Sobolev BGG diagram
\[
\begin{tikzcd}
H^2(T;\mathbb R^3)
 \arrow{r}{\grad}
 &
H^1(\curl,T;\mathbb M)
 \arrow{r}{\curl}
 &
H^1(T;\mathbb M)
 \arrow{r}{\div}
 & L^2(T;\mathbb R^3) \to 0
 \\
H^2(T;\mathbb R^3)\arrow[ur,swap,"\mskw"']
 \arrow{r}{\grad}
 &H^1(T;\mathbb M)
 \arrow[ur,swap,"S"'] \arrow{r}{\curl}
 &
H(\div,T;\mathbb M)
 \arrow[ur,swap,"-2\vskw"'] \arrow{r}{\div} \arrow[r]
 & L^2(T;\mathbb R^3)\to0.
\end{tikzcd}
\]

Applying Proposition~2.3 of
\cite{ChristiansenGopalakrishnanGuzmanHu2024} to this anticommutative diagram
yields the following exact sequence for $k\geq1$:
\begin{equation}\label{163}
\begin{aligned}
\begin{bmatrix}
    V_{k+3}^{\hess}(T^{\rm R};\mathbb R^3) \\
    V_{k+2}^{\hess}(T^{\rm R};\mathbb R^3)
\end{bmatrix}
&\xrightarrow{[\grad, -\!\mskw]}
\Sigma_{k+2}^{1,\curl}(T;\mathbb M)
\xrightarrow{\curl S^{-1} \curl }
\Sigma_k^{\div}(T;\mathbb M)
\\
&\xrightarrow{ \begin{bmatrix}
    2 \vskw \\
    \div 
\end{bmatrix}}
\begin{bmatrix}
\mathbb P_{k}^{-1}(T^{\rm R};\mathbb R^3) \\
 \mathbb P_{k-1}^{-1}(T^{\rm R};\mathbb R^3)
 \end{bmatrix}
 \to 0.
\end{aligned}
\end{equation}

We first record two consequences.

\begin{corollary}\label{cor:vskwonto}
For $k\geq 1$ and $T\in\mathcal T_h$, we have
\begin{equation}\label{eq:vskwonto}
\vskw \Sigma_k^{\div}(T;\mathbb M)
=
\mathbb P_{k}^{-1}(T^{\rm R};\mathbb R^3).
\end{equation}
Consequently, the dimension formula \eqref{eq:SigmmadivSdim} holds.
\end{corollary}
\begin{proof}
The identity \eqref{eq:vskwonto} follows immediately from the exact sequence
\eqref{163}.  Since (cf. \cite[p.~1064]{FuGuzmanNeilan2020})
\begin{equation}\label{eq:VkHdivdim}
\dim V_k^{\div}(T^{\rm R})=(k+1)(k+2)(2k+3),
\end{equation}
the dimension formula \eqref{eq:SigmmadivSdim} follows from \eqref{eq:vskwonto}.
\end{proof}

\begin{corollary}\label{cor:localelascomplex3d_lemma}
The following complex is exact:
\begin{equation}\label{eq:localelascomplex3d_lemma}
\begin{aligned}
{\rm RM}
\xrightarrow{\subset}
V_{k+3}^{\hess}(T^{\rm R};\mathbb R^3)
\xrightarrow{\defm}&
\sym (\Sigma_{k+2}^{1,\curl}(T;\mathbb M))
\\
&\xrightarrow{\inc}
\Sigma_k^{\div}(T;\mathbb S)
\xrightarrow{\div}
\mathbb P_{k-1}^{-1}(T^{\rm R};\mathbb R^3)
\to 0.
\end{aligned}
\end{equation}
\end{corollary}

\begin{proof}
Let
$\boldsymbol{v}\in\mathbb P_{k-1}^{-1}(T^{\rm R};\mathbb R^3)$.
By the exactness of \eqref{163}, there exists
$\boldsymbol{\sigma}\in\Sigma_k^{\div}(T;\mathbb M)$ such that
$\dive\boldsymbol{\sigma}=\boldsymbol{v}$ and
$\vskw\boldsymbol{\sigma}=0$.  Hence
$\boldsymbol{\sigma}\in\Sigma_k^{\div}(T;\mathbb S)$.

Next let $\boldsymbol{\sigma}\in \Sigma_k^{\div}(T;\mathbb S)$ and suppose
that $\dive\boldsymbol{\sigma}=0$.  Again by the exactness of \eqref{163},
there exists $\boldsymbol{\tau}\in \Sigma_{k+2}^{1,\curl}(T;\mathbb M)$ such
that $\boldsymbol{\sigma}=\curl S^{-1} \curl\boldsymbol{\tau}$.
Then
\begin{equation*}
\boldsymbol{\sigma}
= \curl S^{-1} \curl (\sym\boldsymbol{\tau})
=  \inc (\sym\boldsymbol{\tau})
\in \inc\sym (\Sigma_{k+2}^{1,\curl}(T;\mathbb M)).
\end{equation*}

Finally, let $\boldsymbol{\tau}\in \Sigma_{k+2}^{1,\curl}(T;\mathbb M)$
satisfy $\inc(\sym\boldsymbol{\tau})=0$.  Then
$\curl S^{-1}\curl\boldsymbol{\tau}=0$.  By the exactness of \eqref{163},
there exist
$\boldsymbol{v}\in V_{k+3}^{\hess}(T^{\rm R};\mathbb R^3)$ and
$\boldsymbol{w}\in V_{k+2}^{\hess}(T^{\rm R};\mathbb R^3)$ such that
$\boldsymbol{\tau}= \grad\boldsymbol{v}-\!\mskw\boldsymbol{w}$.
Consequently, $\defm\boldsymbol{v}=\sym\boldsymbol{\tau}$, as required.
\end{proof}

\begin{lemma}\label{183}
The sequence \eqref{eq:localelascomplex3d} is exact for $k\geq1$.
\end{lemma}
\begin{proof}
By the exact sequence \eqref{eq:localelascomplex3d_lemma}, it is enough to
identify the symmetric part of the middle space:
    \begin{equation*}
        \sym (\Sigma_{k+2}^{1,\curl}(T;\mathbb M))= \Sigma_{k+2}^{1,\inc}(T;\mathbb S).
    \end{equation*}

The inclusion
$\sym(\Sigma_{k+2}^{1,\curl}(T;\mathbb M))
\subseteq \Sigma_{k+2}^{1,\inc}(T;\mathbb S)$
is immediate.  Conversely, let
$\boldsymbol{\tau}\in\Sigma_{k+2}^{1,\inc}(T;\mathbb S)$ and set
$\boldsymbol{\sigma}=\curl S^{-1}\curl\boldsymbol{\tau}\in L^2(T;\mathbb S)$.
Then $\boldsymbol{\sigma}\in\Sigma_k^{\div}(T;\mathbb M)$,
$\vskw \boldsymbol{\sigma}=0$, and
$\dive \boldsymbol{\sigma}=0$.  By the exactness of \eqref{163}, there exists
$\boldsymbol{\omega}\in\Sigma_{k+2}^{1,\curl}(T;\mathbb M)$ such that
$\boldsymbol{\sigma}=\curl S^{-1}\curl\boldsymbol{\omega}$.

Let
$\boldsymbol{q}=S^{-1}\curl(\boldsymbol{\tau}-\boldsymbol{\omega})$.
Then
$\boldsymbol{q}\in \mathbb P_{k+1}^{-1}(T^{\rm R};\mathbb M)
\cap H(\curl,T;\mathbb M)$,
$\curl\boldsymbol{q}=0$, and $\boldsymbol{q}$ is continuous at the vertices
of $T^{\rm R}$.  Therefore, there exists
$\boldsymbol{v}\in \mathbb P_{k+2}^{\grad}(T^{\rm R};\mathbb R^3)$ in the
vector Lagrange space such that $\grad\boldsymbol{v}=\boldsymbol{q}$.
In particular, $\boldsymbol{v}$ is $C^1$ at the vertices of $T^{\rm R}$.
Set $\boldsymbol{\theta}=\boldsymbol{\tau}+\mskw\boldsymbol{v}$.  By
\eqref{eq:anticommutprop2.1},
\begin{equation*}
    \curl\boldsymbol{\theta}
    = \curl\boldsymbol{\tau} + \curl(\mskw\boldsymbol{v})
    = \curl\boldsymbol{\tau} - S\grad\boldsymbol{v}
    = \curl\boldsymbol{\omega}.
\end{equation*}
Thus $\boldsymbol{\theta}\in \Sigma_{k+2}^{1,\curl}(T;\mathbb M)$.  Since
$\boldsymbol{\tau}=\sym\boldsymbol{\theta}$, the reverse inclusion follows.
\end{proof}

The dimension of the incompatibility space
$\Sigma_{k+2}^{1,\inc}(T;\mathbb S)$ now follows from the exact sequence
\eqref{eq:localelascomplex3d}, thereby justifying formula
\eqref{eq:Sigmma1incSdim}.

\begin{lemma}\label{lm:Sigmakincdim1}
The dimension formula \eqref{eq:Sigmma1incSdim} for $k\geq1$ holds.
\end{lemma}
\begin{proof}
Using the exact complex \eqref{eq:localelascomplex3d}, together with
\eqref{eq:Vhessdim}, \eqref{eq:SigmmadivSdim},
$\dim\mathbb P_{k-1}^{-1}(T^{\rm R};\mathbb R^3)=2k(k+1)(k+2)$,
and $\dim{\rm RM}=6$, we obtain
\begin{align*}
\dim\Sigma_{k+2}^{1,\inc}(T;\mathbb S)
&=\dim V_{k+3}^{\hess}(T^{\rm R};\mathbb R^3)
 +\dim\Sigma_k^{\div}(T;\mathbb S)-2k(k+1)(k+2)-6 \\
&=2(k^3+6k^2+20k+30) + (4k+3)(k+1)(k+2) \\
&\quad -2k(k+1)(k+2)-6 = 4k^3+21k^2+53k+60.
\end{align*}
Hence, \eqref{eq:Sigmma1incSdim} holds.
\end{proof}

\subsection{Exactness of the lower-regularity local complex}
We next prove the exactness of \eqref{eq:localelascomplex13d}.  The
lower-regularity complex is obtained from an analogous local BGG diagram.  The
row-wise de Rham complexes now give
\begin{equation}\label{eq:localbggdiagram_2}
\begin{tikzcd}[column sep=small]
V_{k+3}^{1,\curl}(T^{\rm R}) \arrow{r}{\grad} & \Sigma_{k+2}^{\curl,\skw}(T;\mathbb M)
 \arrow{r}{\curl}
 &
\mathbb P_{k+1}^{\grad}(T^{\rm R};\mathbb M)
 \arrow{r}{\div}
 & \mathbb P_{k}^{-1}(T^{\rm R};\mathbb R^3) \to 0
 \\
V_{k+2}^{\hess}(T^{\rm R};\mathbb R^3)\arrow[ur,swap,"\mskw"'] \arrow{r}{\grad}  & 
\mathbb P_{k+1}^{\grad}(T^{\rm R};\mathbb M)
 \arrow[ur,swap,"S"'] \arrow{r}{\curl}
 &
\Sigma_k^{\div}(T;\mathbb M)
 \arrow[ur,swap,"-2\vskw"'] \arrow{r}{\div} \arrow[r]
 & \mathbb P_{k-1}^{-1}(T^{\rm R};\mathbb R^3)\to0,
\end{tikzcd}
\end{equation}
where
\begin{align*}
\Sigma_{k+2}^{\curl,\skw}(T;\mathbb M)
:= & \{\boldsymbol{\tau}\in
\mathbb P_{k+2}^{-1}(T^{\rm R};\mathbb M) \cap H(\curl, T;\mathbb M):
\vskw\boldsymbol{\tau}\in H^1(T;\mathbb R^3), \\
&\quad \curl\boldsymbol{\tau}\in H^1(T; \mathbb M),\ 
\boldsymbol{\tau}\text{ is continuous at all vertices of }T^{\rm R}\}.
\end{align*}

\begin{lemma}
The top sequence in \eqref{eq:localbggdiagram_2} is exact.    
\end{lemma}
\begin{proof}
The exactness at the last two terms follows from the exactness of the top
sequence in \eqref{eq:localbggdiagram_1}.  Let
$\boldsymbol{\sigma} \in \mathbb P_{k+1}^{\grad}(T^{\rm R};\mathbb M)$ satisfy
$\div\boldsymbol{\sigma}=0$.  Applying the exactness of the top sequence in
\eqref{eq:localbggdiagram_1} again, there exists
$\boldsymbol{\tau} \in \Sigma_{k+2}^{1,\curl}(T;\mathbb M)
\subseteq \Sigma_{k+2}^{\curl,\skw}(T;\mathbb M)$
such that $\curl\boldsymbol{\tau}=\boldsymbol{\sigma}$.

It remains to prove exactness at
$\Sigma_{k+2}^{\curl,\skw}(T;\mathbb M)$.  Suppose that
$\boldsymbol{\tau}\in \Sigma_{k+2}^{\curl,\skw}(T;\mathbb M)$ satisfies
$\curl\boldsymbol{\tau}=0$.  Then there exists
$\boldsymbol{v} \in \mathbb P_{k+3}^{\grad}(T^{\rm R};\mathbb R^3)$ such that
$\grad\boldsymbol{v}=\boldsymbol{\tau}$.  By identity (3.2.1) in
\cite{ChenHuang2022c}, or equivalently by the anticommutativity of the BGG
diagram (13) in \cite{ChenHuang2026},
\[
\curl\boldsymbol{v}
= 2\vskw \grad\boldsymbol{v}
= 2\vskw\boldsymbol{\tau}
\in H^1(T;\mathbb R^3).
\]
Therefore, $\boldsymbol{v}\in V_{k+3}^{1,\curl}(T^{\rm R})$.
\end{proof}

Applying Proposition~2.3 of
\cite{ChristiansenGopalakrishnanGuzmanHu2024} to
\eqref{eq:localbggdiagram_2} yields the following exact sequence for
$k\geq1$:
\begin{equation}\label{263}
\begin{aligned}
\begin{bmatrix}
    V_{k+3}^{1,\curl}(T^{\rm R}) \\
    V_{k+2}^{\hess}(T^{\rm R};\mathbb R^3)
\end{bmatrix}
&\xrightarrow{[\grad, -\!\mskw]}
\Sigma_{k+2}^{\curl,\skw}(T;\mathbb M)
\xrightarrow{\curl S^{-1} \curl }
\Sigma_k^{\div}(T;\mathbb M)
\\
&\xrightarrow{ \begin{bmatrix}
    2 \vskw \\
    \div 
\end{bmatrix}}
\begin{bmatrix}
\mathbb P_{k}^{-1}(T^{\rm R};\mathbb R^3) \\
 \mathbb P_{k-1}^{-1}(T^{\rm R};\mathbb R^3)
 \end{bmatrix}
 \to 0.
\end{aligned}
\end{equation}

\begin{lemma}
The sequence \eqref{eq:localelascomplex13d} is exact for $k\geq1$.
\end{lemma}

\begin{proof}
The exact sequence \eqref{263} gives the symmetric sequence
\begin{equation*}
\begin{aligned}
{\rm RM}
\xrightarrow{\subset}
 V_{k+3}^{1,\curl}(T^{\rm R})
\xrightarrow{\defm}&
\sym (\Sigma_{k+2}^{\curl,\skw}(T;\mathbb M))
\\
&\xrightarrow{\inc}
\Sigma_k^{\div}(T;\mathbb S)
\xrightarrow{\div}
\mathbb P_{k-1}^{-1}(T^{\rm R};\mathbb R^3)
\to 0,
\end{aligned}
\end{equation*}
by the same argument as in the proof of
Corollary~\ref{cor:localelascomplex3d_lemma}.  Moreover,
\begin{alignat}{1}
\sym (\Sigma_{k+2}^{\curl,\skw}(T;\mathbb M))= \Sigma_{k+2}^{\inc^+}(T;\mathbb S).
\end{alignat}
This identification follows by repeating the argument in the proof of
Lemma~\ref{183}.  Hence the sequence \eqref{eq:localelascomplex13d} is exact.
\end{proof}

\begin{lemma}\label{lm:Sigmakincdim2}
The dimension formula \eqref{eq:Sigmmainc+Sdim} for $k\geq1$ holds.
\end{lemma}
\begin{proof}
From the exact complex \eqref{eq:localelascomplex13d} and formulas
\eqref{eq:Vkp2H1curldim} and \eqref{eq:SigmmadivSdim}, using the same
dimensions for the last term and for ${\rm RM}$ as above, we obtain
\begin{align*}
\dim\Sigma_{k+2}^{\inc^+}(T;\mathbb S)
&=\dim V_{k+3}^{1,\curl}(T^{\rm R})
 +\dim\Sigma_k^{\div}(T;\mathbb S)-2k(k+1)(k+2)-6 \\
&=(k+3)(2k^2+9k+22) + (4k+3)(k+1)(k+2) \\
&\quad -2k(k+1)(k+2)-6 = 4k^3+24k^2+62k+66.
\end{align*}
Therefore, \eqref{eq:Sigmmainc+Sdim} holds.
\end{proof}

\subsection{Bubble complexes and dimension formulas}\label{subsec:bubble-results}
We collect the bubble exactness and dimension results needed for the finite
element constructions below.  The proofs are based on bubble de Rham complexes
and the local BGG construction, and are given in
Appendix~\ref{app:bubble-complexes}.

For $k\geq1$, define the symmetric $H(\div)$ bubble space by
\[
\mathbb B_k^{\div}(T^{\rm R};\mathbb S)
:=H_0(\div,T;\mathbb S)\cap\Sigma_k^{\div}(T;\mathbb S)
=H_0(\div,T;\mathbb S)\cap\mathbb P_k^{-1}(T^{\rm R};\mathbb S).
\]

\begin{lemma}\label{lem:bubble-divergence}
For $k\geq1$ and $T\in\mathcal T_h$,
\begin{equation}\label{divontoB}
\div\mathbb B_k^{\div}(T^{\rm R};\mathbb S)
=\mathbb P_{k-1}^{-1}(T^{\rm R};\mathbb R^3)/{\rm RM},
\end{equation}
and
\begin{equation}\label{eq:bubbledivSdim}
\begin{aligned}
\dim\mathbb B_k^{\div}(T^{\rm R};\mathbb S)
&=\dim\mathbb B_k^{\div}(T^{\rm R};\mathbb M)
 -\dim\mathbb P_k^{-1}(T^{\rm R};\mathbb R^3)+3\\
&=(k+1)(k+2)(4k-3).
\end{aligned}
\end{equation}
\end{lemma}

We first record the bubble results associated with the smoother local complex
\eqref{eq:localelascomplex3d}.  Define
\begin{align*}
\mathbb B_{k+3}^{\rm herm}(T^{\rm R})
&:=V_{k+3}^{\hess}(T^{\rm R})\cap H_0^1(T),\\
\mathbb B_{k+2}^{\inc}(T^{\rm R};\mathbb S)
&:=\{\boldsymbol\tau\in\Sigma_{k+2}^{1,\inc}(T;\mathbb S):
\tr_1(\boldsymbol\tau)=0,\ \tr_2(\boldsymbol\tau)=0,\\
&\qquad \boldsymbol\tau\text{ and }\nabla\boldsymbol\tau
\text{ vanish at the vertices of }T,\\
&\qquad \boldsymbol\tau\text{ and }
(\curl\boldsymbol\tau)^{\intercal}\boldsymbol t
\text{ vanish on the edges of }T\}.
\end{align*}

\begin{lemma}\label{lem:bubble-complex-smooth}
For $k\geq1$ and $T\in\mathcal T_h$, the bubble elasticity complex
\begin{equation}\label{eq:bubbleelascomplex3d}
\mathbb B_{k+3}^{\rm herm}(T^{\rm R};\mathbb R^3)
\xrightarrow{\defm}
\mathbb B_{k+2}^{\inc}(T^{\rm R};\mathbb S)
\xrightarrow{\inc}
\mathbb B_k^{\div}(T^{\rm R};\mathbb S)
\xrightarrow{\div}
\mathbb P_{k-1}^{-1}(T^{\rm R};\mathbb R^3)/{\rm RM}
\to0
\end{equation}
is exact.
\end{lemma}

\begin{lemma}\label{lem:bubble-dimensions-smooth}
For $k\geq1$,
\begin{align}
\dim\mathbb B_{k+3}^{\rm herm}(T^{\rm R})
&=\frac{2}{3}k(k+1)(k+2),\label{eq:bubblehermdim}\\
\dim\mathbb B_{k+2}^{\inc}(T^{\rm R};\mathbb S)
&=4k^3+9k^2-k.\label{eq:BincSdim}
\end{align}
\end{lemma}

We next record the corresponding results for the lower-regularity local complex
\eqref{eq:localelascomplex13d}.  Define
\begin{align*}
\mathbb B_{k+3}^{1,\curl}(T^{\rm R})
&:=\{\boldsymbol v\in V_{k+3}^{1,\curl}(T^{\rm R}):
\boldsymbol v\text{ and }\curl\boldsymbol v
\text{ vanish on }\partial T\},\\
\mathbb B_{k+2}^{\inc^+}(T^{\rm R};\mathbb S)
&:=\{\boldsymbol\tau\in\Sigma_{k+2}^{\inc^+}(T;\mathbb S):
\boldsymbol\tau\text{ vanishes at all vertices of }T,\\
&\qquad \boldsymbol\tau\times\boldsymbol n
\text{ and }(\curl\boldsymbol\tau)^{\intercal}\times\boldsymbol n
\text{ vanish on }\partial T\}.
\end{align*}

\begin{lemma}\label{lem:bubble-complex-less-regular}
For $k\geq1$ and $T\in\mathcal T_h$, the bubble elasticity complex
\begin{equation}\label{eq:bubbleelascomplex13d}
\mathbb B_{k+3}^{1,\curl}(T^{\rm R})
\xrightarrow{\defm}
\mathbb B_{k+2}^{\inc^+}(T^{\rm R};\mathbb S)
\xrightarrow{\inc}
\mathbb B_k^{\div}(T^{\rm R};\mathbb S)
\xrightarrow{\div}
\mathbb P_{k-1}^{-1}(T^{\rm R};\mathbb R^3)/{\rm RM}
\to0
\end{equation}
is exact.
\end{lemma}

\begin{lemma}\label{lem:bubble-dimensions-less-regular}
For $k\geq1$,
\begin{align}
\dim\mathbb B_{k+3}^{1,\curl}(T^{\rm R})
&=k(k+1)(2k+3),\label{eq:bubble1curldim}\\
\dim\mathbb B_{k+2}^{\inc^+}(T^{\rm R};\mathbb S)
&=4k^3+8k^2-2k.\label{eq:Binc+Sk2}
\end{align}
\end{lemma}

\section{Finite element complex for the $H^1(\curl)$--$H(\inc^+)$ elasticity sequence}\label{sec:femelascomplexinc+}

This section constructs a finite element elasticity complex for the
$H^1(\curl)$--$H(\inc^+)$ elasticity sequence on the Alfeld refinement of a
tetrahedral mesh.  For $k\geq1$, the discrete sequence is
\begin{equation}\label{eq:elascomplex13d}
{\rm RM}
\xrightarrow{\subset}
V_h^{1,\curl}
\xrightarrow{\defm}
\Sigma_h^{\inc^+}
\xrightarrow{\inc}
\Sigma_{k,h}^{\div}\xrightarrow{\div} V_{k-1,h}^{L^2}
\to 0.
\end{equation}
Here
\[
V_{k-1,h}^{L^2}
:=\{\boldsymbol v_h\in L^2(\Omega;\mathbb R^3):
\boldsymbol v_h|_T\in\mathbb P_{k-1}^{-1}(T^{\rm R};\mathbb R^3)
\textrm{ for each } T\in\mathcal T_h\},
\]
and the finite element spaces $V_h^{1,\curl}$, $\Sigma_h^{\inc^+}$, and
$\Sigma_{k,h}^{\div}$ are defined in
\eqref{eq:Vh1curlk1}, \eqref{eq:Sigmahinc+k1}, and
\eqref{eq:Sigmahdiv}, respectively.  The sequence
\eqref{eq:elascomplex13d} is a finite element subcomplex of the continuous
elasticity complex \eqref{elascomplex:H1curlinc+}.
It should be distinguished from the finite element elasticity complex of
\cite{ChristiansenGopalakrishnanGuzmanHu2024}, which is a subcomplex of the
smoother elasticity complex
\[
{\rm RM}
\xrightarrow{\subset}
H^2(\Omega;\mathbb R^3)
\xrightarrow{\defm}
H^1(\inc,\Omega;\mathbb S)
\xrightarrow{\inc}
H(\div,\Omega;\mathbb S)
\xrightarrow{\div}
L^2(\Omega;\mathbb R^3)
\to 0.
\]
Commuting interpolation operators for \eqref{eq:elascomplex13d} are also
constructed.

\subsection{Finite elements for tensors on faces}
Two face finite elements are used as trace elements in the three-dimensional
construction.

Let $F$ be a triangular face and identify tangential tensors on $F$
with two-dimensional matrices in a fixed tangential frame.  The first face
element controls $\tr_1(\boldsymbol\tau)=\Pi_F\boldsymbol\tau\Pi_F$.  For
$k\geq1$, it has shape space
$\mathbb P_{k+2}(F;\mathbb S_F)$ and is conforming for
$H(\rot_F\rot_F,F;\mathbb S_F)$, where
$\mathbb S_F:=\Pi_F\mathbb S\Pi_F$.  Its degrees of freedom are
\begin{subequations}\label{Hrotrotfem2d1dof}
\begin{align}
\boldsymbol{\tau}(\delta), & \quad\delta\in \Delta_0(F), \label{Hrotrotfem2d1dof1}\\
(\boldsymbol{\tau}\boldsymbol{t}, \boldsymbol  q)_e, & \quad\boldsymbol  q\in\mathbb P_{k}(e;\mathbb R^2),  e\in\Delta_1(F),\label{Hrotrotfem2dd1of2}\\
(\boldsymbol t^{\intercal}\rot_F\boldsymbol{\tau}, q)_e, & \quad q\in\mathbb P_{k+1}(e),  e\in\Delta_1(F),\label{Hrotrotfem2d1dof3}\\
(\rot_F\rot_F\boldsymbol{\tau}, q)_F, & \quad q\in\mathbb P_{k}(F)/\mathbb P_{1}(F),\label{Hrotrotfem2d1dof4}\\
(\rot_F(\boldsymbol\tau), \boldsymbol q)_F,
& \quad
\boldsymbol q\in \boldsymbol x(\mathbb P_{k-1}(F)/\mathbb R),
\label{Hrotrotfem2d1dof5}\\
(\boldsymbol\tau, \boldsymbol q)_F,
& \quad
\boldsymbol q\in (\boldsymbol x\otimes\boldsymbol x)\mathbb P_{k-2}(F). \label{Hrotrotfem2d1dof6}
\end{align}
\end{subequations}

\begin{lemma}\label{lem:unisolventHrotrotfem2d1}
The degrees of freedom \eqref{Hrotrotfem2d1dof} are unisolvent for
$\mathbb P_{k+2}(F;\mathbb S_F)$.
\end{lemma}
\begin{proof}
The result follows by rotation from the $H(\div)$- and
$H(\div\div)$-conforming finite elements for symmetric tensors in
\cite[Theorem~2.4]{HuMaZhang2021} and
\cite[Theorem~5.5]{ChenHuang2022b}.
\end{proof}

The second face element controls
$\Pi_F((\curl\boldsymbol\tau)^{\intercal})\Pi_F$.  It is an
$H(\rot_F,F;\mathbb M_F)$ element with
$\mathbb M_F:=\Pi_F\mathbb M\Pi_F$.  Its shape function space is
$\mathbb P_{k+1}(F;\mathbb M_F)$, and its degrees of freedom are
\begin{subequations}\label{Hrotfem2d1dof}
\begin{align}
(\boldsymbol{\tau}\boldsymbol t, \boldsymbol q)_e, & \quad\boldsymbol  q\in\mathbb P_{k+1}(e; \mathbb R^2),  e\in\Delta_1(F),\label{Hrotfem2d1dof2}\\
(\rot_F\boldsymbol  \tau, \boldsymbol  q)_F, & \quad\boldsymbol  q\in\mathbb P_{k}(F;\mathbb R^2)/\mathbb R^2,\label{Hrotfem2d1dof3}\\
(\boldsymbol  \tau, \boldsymbol  q)_F, & \quad\boldsymbol  q\in \mathbb P_{k-1}(F;\mathbb R^2)\otimes\boldsymbol{x}. \label{Hrotfem2d1dof4}
\end{align}
\end{subequations}

\begin{lemma}\label{lem:unisolventHrotfem2d1}
The degrees of freedom \eqref{Hrotfem2d1dof} are unisolvent for
$\mathbb P_{k+1}(F;\mathbb M_F)$.
\end{lemma}
\begin{proof}
The result follows by rotation from the tensor-valued
Brezzi-\allowbreak Douglas-\allowbreak Marini element of
\cite{BrezziDouglasMarini1985}.
\end{proof}

\subsection{$H(\div;\mathbb S)$-conforming finite elements}
We first define the final stress space in \eqref{eq:elascomplex13d}.  For
$k\geq1$, the local $H(\div)$-conforming space for symmetric tensor fields on
each tetrahedron $T$ is $\Sigma_k^{\div}(T;\mathbb S)$; see
\cite{ChenHuang2025,GopalakrishnanGuzmanLee2025,Krizek1982}.
The degrees of freedom are given by
\begin{subequations}\label{HdivSfemdof}
\begin{align}
(\boldsymbol{\tau}\boldsymbol n,\boldsymbol q)_F,
& \quad \boldsymbol q\in\mathbb P_k(F;\mathbb R^3),\quad F\in\Delta_2(T),\label{HdivSfemdof1}\\
(\div\boldsymbol\tau,\boldsymbol q)_T,
& \quad \boldsymbol q\in\mathbb P_{k-1}^{-1}(T^{\rm R};\mathbb R^3)/{\rm RM},\label{HdivSfemdof2}\\
(\boldsymbol\tau,\boldsymbol q)_T,
& \quad \boldsymbol q\in\mathbb B_k^{\div}(T^{\rm R};\mathbb S)\cap\ker(\div).
\label{HdivSfemdof3}
\end{align}
\end{subequations}
These degrees of freedom differ from those in \cite[(22)]{ChenHuang2025};
they are chosen to facilitate the construction of a commuting projection
operator.

\begin{lemma}
The degrees of freedom \eqref{HdivSfemdof} are unisolvent for
$\Sigma_k^{\div}(T;\mathbb S)$.
\end{lemma}
\begin{proof}
By \eqref{eq:vskwonto} and \eqref{eq:bubbledivSdim},
\begin{align*}
\dim \Sigma_k^{\div}(T;\mathbb S)
&= \dim\Sigma_k^{\div}(T;\mathbb M)
 - 3\dim \mathbb P_{k}^{-1}(T^{\rm R}), \\
\dim \mathbb B_k^{\div}(T^{\rm R};\mathbb S)
&= \dim\mathbb B_k^{\div}(T^{\rm R};\mathbb M)
 - 3\dim \mathbb P_{k}^{-1}(T^{\rm R})+3.
\end{align*}
Therefore,
\begin{align*}
\dim \Sigma_k^{\div}(T;\mathbb S)
&= \dim \mathbb B_k^{\div}(T^{\rm R};\mathbb S)
 + \dim\Sigma_k^{\div}(T;\mathbb M)  - \dim\mathbb B_k^{\div}(T^{\rm R};\mathbb M)-3.
\end{align*}
Together with \eqref{divontoB}, this shows that the number of degrees of
freedom in
\eqref{HdivSfemdof} equals $\dim \Sigma_k^{\div}(T;\mathbb S)$.

Assume $\boldsymbol{\tau}\in\Sigma_k^{\div}(T;\mathbb S)$ and that all
degrees of freedom in \eqref{HdivSfemdof} vanish.  The vanishing of
\eqref{HdivSfemdof1} implies
$\boldsymbol{\tau}\in \mathbb B_k^{\div}(T^{\rm R};\mathbb S)$.  Using
\eqref{divontoB} again, the vanishing of
\eqref{HdivSfemdof2}--\eqref{HdivSfemdof3} yields
$\boldsymbol{\tau}=0$.
\end{proof}

The global $H(\div;\mathbb S)$-conforming finite element space is defined by
\begin{equation}\label{eq:Sigmahdiv}
\Sigma_{k,h}^{\div}
=\{\boldsymbol\tau_h\in H(\div,\Omega;\mathbb S):
\boldsymbol\tau_h|_T\in\Sigma_k^{\div}(T;\mathbb S)
\text{ for all }T\in\mathcal T_h\}.
\end{equation}
Let $I_h^{\div}:H^1(\Omega;\mathbb S)\to\Sigma_{k,h}^{\div}$ be the
interpolation operator determined by \eqref{HdivSfemdof}, and let $Q_h$ denote
the $L^2$ projection onto $V_{k-1,h}^{L^2}$.

\begin{lemma}
For $k\geq1$,
\begin{equation}\label{eq:commutativitydiv}
\div(I_h^{\div}\boldsymbol\tau)=Q_h(\div\boldsymbol\tau),
\qquad \forall\,\boldsymbol\tau\in H^1(\Omega;\mathbb S).
\end{equation}
\end{lemma}
\begin{proof}
Let $T\in\mathcal T_h$.  For $\boldsymbol q\in{\rm RM}$, integration by parts
and the face degree of freedom \eqref{HdivSfemdof1} give
\[
(\div(\boldsymbol\tau-I_h^{\div}\boldsymbol\tau),\boldsymbol q)_T=0.
\]
Together with \eqref{HdivSfemdof2}, this proves the desired commutativity.
\end{proof}

\begin{lemma}
For $k\geq 1$, we have 
\begin{equation}\label{divonto}
\div\Sigma_{k,h}^{\div}=V_{k-1,h}^{L^2}.
\end{equation}
\end{lemma}
\begin{proof}
The inclusion ``$\subseteq$'' is clear from the definition of the local space.
For the reverse inclusion, use a standard right inverse of the divergence on
symmetric $H^1$ tensor fields: for each
$\boldsymbol v_h\in V_{k-1,h}^{L^2}$ choose
$\boldsymbol\tau\in H^1(\Omega;\mathbb S)$ with
$\div\boldsymbol\tau=\boldsymbol v_h$.  Then
\eqref{eq:commutativitydiv} gives
$\div(I_h^{\div}\boldsymbol\tau)=Q_h\boldsymbol v_h=\boldsymbol v_h$.
\end{proof}

\subsection{$H^1(\curl)$-conforming finite elements for vector fields}
The first nontrivial space in \eqref{eq:elascomplex13d} is the
$H^1(\curl)$-conforming vector space.  On each $T$, we take the local shape
function space $V_{k+3}^{1,\curl}(T^{\rm R})$.  Its degrees of freedom are
\begin{subequations}\label{H1femk1dof}
\begin{align}
\boldsymbol{v}(\delta),\ \nabla \boldsymbol{v}(\delta),
& \quad
\delta\in \Delta_0(T),
\label{H1femk1dof1}\\
(\boldsymbol{v}, \boldsymbol{q})_e,
& \quad
\boldsymbol q\in\mathbb P_{k-1}(e;\mathbb R^3),\;
e\in\Delta_1(T),
\label{H1femk1dof2}\\
(\curl\boldsymbol{v}, \boldsymbol{q})_e,
& \quad
\boldsymbol q\in\mathbb P_{k}(e;\mathbb R^3),\;
e\in\Delta_1(T),
\label{H1femk1dof3}\\
(\grad_F(\Pi_F(\curl\boldsymbol{v})), \boldsymbol q)_F,
& \quad
\boldsymbol q\in \mathbb P_{k-1}(F;\mathbb R^2)\otimes\boldsymbol x,\;
F\in\Delta_2(T),
\label{H1femk1dof4}\\
(\grad_F(\rot_F\boldsymbol{v}), \boldsymbol q)_F,
& \quad
\boldsymbol q\in \boldsymbol x(\mathbb P_{k-1}(F)/\mathbb R),\;
F\in\Delta_2(T),
\label{H1femk1dof5}\\
(\grad_F(\Pi_F\boldsymbol{v}), \boldsymbol q)_F,
& \quad
\boldsymbol q\in (\boldsymbol x\otimes\boldsymbol x)\mathbb P_{k-2}(F),\;
F\in\Delta_2(T),
\label{H1femk1dof6}\\
(\Pi_F\defm(\boldsymbol{v})\boldsymbol{n}, \boldsymbol q)_F,
& \quad
\boldsymbol q\in \boldsymbol x \mathbb P_{k}(F),\;
F\in\Delta_2(T),
\label{H1femk1dof7}\\
(\defm(\boldsymbol{v}), \boldsymbol{q})_T,
& \quad
\boldsymbol{q}\in
\defm\bigl(\mathbb B_{k+3}^{1,\curl}(T^{\rm R})\bigr).
\label{H1femk1dof8}
\end{align}
\end{subequations}
The degrees of freedom in \eqref{H1femk1dof} differ from those given in
\cite[(4.12)]{FuGuzmanNeilan2020}.

\begin{lemma}
The degrees of freedom \eqref{H1femk1dof} are unisolvent for
$V_{k+3}^{1,\curl}(T^{\rm R})$.
\end{lemma}

\begin{proof}
By \eqref{eq:bubble1curldim},
the number of degrees of freedom is
\[
\begin{aligned}
&48+18(2k+1)+2(5k^2+5k)+k(k+1)(2k+3) =2k^3+15k^2+49k+66,
\end{aligned}
\]
which equals $\dim V_{k+3}^{1,\curl}(T^{\rm R})$ by
\eqref{eq:Vkp2H1curldim}.  Hence it remains to prove uniqueness.

Let $\boldsymbol v\in V_{k+3}^{1,\curl}(T^{\rm R})$ and suppose that all degrees
of freedom in \eqref{H1femk1dof} vanish.  From
\eqref{H1femk1dof1}--\eqref{H1femk1dof3}, we obtain
\begin{equation}\label{eq:vectorvanishedge0}
\boldsymbol v|_e=0,
\qquad
(\curl\boldsymbol v)|_e=0,
\qquad
\forall\, e\in\Delta_1(T).
\end{equation}
Consequently, for each $F\in\Delta_2(T)$,
\begin{equation}\label{eq:vectorvanishedge}
\Pi_F(\curl\boldsymbol{v})=0,
\qquad
\rot_F\boldsymbol{v}=0
\quad\text{on }\partial F.
\end{equation}
Using integration by parts on each face, the identity
\[
\div_F\bigl(\mathbb P_{k-1}(F;\mathbb R^2)\otimes\boldsymbol x\bigr)
=
\mathbb P_{k-1}(F;\mathbb R^2)
\]
from \cite[Lemma~3.1]{ChenHuang2022b}, and the vanishing of
\eqref{H1femk1dof4}, we obtain
\[
\Pi_F(\curl\boldsymbol v)=0
\qquad\text{on each } F\in\Delta_2(T).
\]
By \eqref{eq:vectorvanishedge},
\begin{equation*}
(\grad_F(\rot_F\boldsymbol{v}), \boldsymbol{x})_F=-(\rot_F\boldsymbol{v},2)_F=0.
\end{equation*}
This identity, together with the vanishing of
\eqref{H1femk1dof5}, integration by parts on each face, and
$\div_F(\boldsymbol x\mathbb P_{k-1}(F))=\mathbb P_{k-1}(F)$, implies that
\begin{equation}\label{eq:vectorvanishface0}    
(\curl\boldsymbol v)\cdot\boldsymbol{n}=\rot_F\boldsymbol v=0
\qquad\text{on each } F\in\Delta_2(T).
\end{equation}
Hence
\[
\curl\boldsymbol v=0
\qquad\text{on }\partial T.
\]

Since $\curl\boldsymbol v=0$ on $\partial T$, the skew-symmetric part of
$\nabla\boldsymbol v$ vanishes there, and therefore
$\nabla\boldsymbol v=\defm\boldsymbol v$ on $\partial T$.  
Consequently,
\[
\Pi_F\defm(\boldsymbol v)\boldsymbol n
=
\grad_F(\boldsymbol v\cdot\boldsymbol n)
\qquad\text{on }F.
\]
The vanishing of \eqref{H1femk1dof7}, together with
$\div_F(\boldsymbol x \mathbb P_{k}(F))=\mathbb P_{k}(F)$, implies
$\boldsymbol v\cdot\boldsymbol n=0$ on $\partial T$.  

By \eqref{eq:vectorvanishface0} and \eqref{eq:vectorvanishedge0}, on each
face $F$ there is a polynomial $p\in\mathbb P_{k-2}(F)$ such that
$\Pi_F\boldsymbol v=\grad_F(b_F^2p)$.
Integration by parts on each face and the vanishing of
\eqref{H1femk1dof6} then imply $\Pi_F\boldsymbol v=0$.  Hence
$\boldsymbol v=0$ on $\partial T$, and therefore
$\boldsymbol v\in\mathbb B_{k+3}^{1,\curl}(T^{\rm R})$.
The vanishing of \eqref{H1femk1dof8} gives $\boldsymbol v=0$.
\end{proof}

The global vector finite element space is defined by
\begin{equation}\label{eq:Vh1curlk1}
\begin{aligned}
V_h^{1,\curl}
&:=
\bigl\{
\boldsymbol v_h\in H^1(\Omega;\mathbb R^3):
\ \boldsymbol v_h|_T\in V_{k+3}^{1,\curl}(T^{\rm R})
\text{ for all }T\in\mathcal T_h,\\
&\qquad\qquad\qquad \text{the degrees of freedom }
\eqref{H1femk1dof1}\text{--}\eqref{H1femk1dof7}
\text{ are single-valued}
\bigr\}.
\end{aligned}
\end{equation}
We have $V_h^{1,\curl}\subseteq H^1(\curl,\Omega)$.
Let
$
I_h^{1,\curl}:H^3(\Omega;\mathbb R^3)\to V_h^{1,\curl}
$
be the interpolation operator determined by \eqref{H1femk1dof}.

\subsection{$H(\inc^+;\mathbb S)$-conforming finite elements}
We next construct the middle space.  The local shape function space is
$\Sigma_{k+2}^{\inc^+}(T;\mathbb S)$.  Its degrees of freedom are
\begin{subequations}\label{Hincfem13ddof}
\begin{align}
\boldsymbol\tau(\delta),
& \quad
\delta\in \Delta_0(T),
\label{Hincfem13ddof1}\\
(\boldsymbol\tau\boldsymbol{t}, \boldsymbol q)_e,
& \quad
\boldsymbol q\in\mathbb P_{k}(e;\mathbb R^3),\;
e\in\Delta_1(T),
\label{Hincfem13ddof2}\\
((\curl\boldsymbol\tau)^{\intercal}\boldsymbol t, \boldsymbol q)_e,
& \quad
\boldsymbol q\in\mathbb P_{k+1}(e;\mathbb R^3),\;
e\in\Delta_1(T),
\label{Hincfem13ddof3}\\
(\rot_F(\Pi_F((\curl\boldsymbol\tau)^{\intercal})\Pi_F), \boldsymbol q)_F,
& \quad
\boldsymbol q\in\mathbb P_{k}(F;\mathbb R^2)/{\rm RM}(F),\;
F\in\Delta_2(T),
\label{Hincfem13ddof4}\\
(\Pi_F((\curl\boldsymbol\tau)^{\intercal})\Pi_F, \boldsymbol q)_F,
& \quad
\boldsymbol q\in \mathbb P_{k-1}(F;\mathbb R^2)\otimes\boldsymbol x,\;
F\in\Delta_2(T),
\label{Hincfem13ddof5}\\
(\rot_F\rot_F(\tr_1(\boldsymbol\tau)), q)_F,
& \quad q\in\mathbb P_{k}(F)/\mathbb P_{1}(F),\;
F\in\Delta_2(T),
\label{Hincfem13ddof6}\\
(\rot_F(\tr_1(\boldsymbol\tau)), \boldsymbol q)_F,
& \quad
\boldsymbol q\in \boldsymbol x(\mathbb P_{k-1}(F)/\mathbb R),\;
F\in\Delta_2(T),
\label{Hincfem13ddof7}\\
(\tr_1(\boldsymbol\tau), \boldsymbol q)_F,
& \quad
\boldsymbol q\in (\boldsymbol x\otimes\boldsymbol x)\mathbb P_{k-2}(F),\;
F\in\Delta_2(T),
\label{Hincfem13ddof8}\\
(\Pi_F\boldsymbol\tau\boldsymbol n, \boldsymbol q)_F,
& \quad
\boldsymbol q\in \boldsymbol x\mathbb P_{k}(F),\;
F\in\Delta_2(T),
\label{Hincfem13ddof9}\\
(\inc\boldsymbol\tau, \boldsymbol q)_T,
& \quad
\boldsymbol q\in
\mathbb B_k^{\div}(T^{\rm R};\mathbb S)\cap\ker(\div),
\label{Hincfem13ddof10}\\
(\boldsymbol\tau, \boldsymbol q)_T,
& \quad
\boldsymbol q\in
\defm\bigl(\mathbb B_{k+3}^{1,\curl}(T^{\rm R})\bigr).
\label{Hincfem13ddof11}
\end{align}
\end{subequations}

\begin{lemma}\label{lem:unisolHincfem13ddof}
The degrees of freedom \eqref{Hincfem13ddof} are unisolvent for
$\Sigma_{k+2}^{\inc^+}(T;\mathbb S)$.
\end{lemma}

\begin{proof}
By the exact bubble sequence \eqref{eq:bubbleelascomplex13d} and
\eqref{eq:Binc+Sk2},
the number of degrees of freedom is
\[
\begin{aligned}
24+18(2k+3)+4(4k^2+7k-3)+(4k^3+8k^2-2k) =4k^3+24k^2+62k+66,
\end{aligned}
\]
which equals $\dim\Sigma_{k+2}^{\inc^+}(T;\mathbb S)$ by
\eqref{eq:Sigmmainc+Sdim}.  It remains to prove uniqueness.

Let $\boldsymbol\tau\in\Sigma_{k+2}^{\inc^+}(T;\mathbb S)$ and suppose that all
degrees of freedom in \eqref{Hincfem13ddof} vanish.  From
\eqref{Hincfem13ddof1}--\eqref{Hincfem13ddof3}, we have
\begin{equation}\label{eq:tensorvanishedge}
\boldsymbol\tau\boldsymbol t|_e=0,
\qquad
((\curl\boldsymbol\tau)^{\intercal}\boldsymbol t)|_e=0,
\qquad
\forall\,e\in\Delta_1(T).
\end{equation}
Moreover,
\[
\boldsymbol t_{F,e}^{\intercal}\rot_F\boldsymbol\tau
=
\boldsymbol n^{\intercal}(\curl\boldsymbol\tau)^{\intercal}\boldsymbol t_{F,e}
=0
\qquad\text{on }\partial F.
\]
Lemma~\ref{lem:unisolventHrotrotfem2d1} and the vanishing degrees of freedom
\eqref{Hincfem13ddof6}--\eqref{Hincfem13ddof8} yield
\[
\Pi_F\boldsymbol\tau\Pi_F=0
\qquad\text{on each }F\in\Delta_2(T).
\]
It follows that
\begin{equation}\label{eq:curltensortnvanishface}
\boldsymbol n^{\intercal}
(\curl\boldsymbol\tau)^{\intercal}\Pi_F
=
\rot_F(\Pi_F\boldsymbol\tau\Pi_F)
=
0
\qquad\text{on }F.
\end{equation}

We next show that the tangential--tangential trace of
$(\curl\boldsymbol\tau)^{\intercal}$ also vanishes.  Let
$\boldsymbol q\in{\rm RM}(F)$.  Since
$\curl_F\boldsymbol q=c\,\Pi_FI\Pi_F$ for some constant $c$, integration by
parts, \eqref{eq:tensorvanishedge}, and $\tr(\curl\boldsymbol\tau)=0$ give
\begin{equation}\label{eq:20260607}
\begin{aligned}
(\rot_F(\Pi_F((\curl\boldsymbol\tau)^{\intercal})\Pi_F),\boldsymbol q)_F
&=
 c(\Pi_F((\curl\boldsymbol\tau)^{\intercal})\Pi_F,I)_F \\
&=
 -c\int_F\rot_F(\Pi_F\boldsymbol\tau\boldsymbol n)\,\dd S
=0.
\end{aligned}
\end{equation}
Lemma~\ref{lem:unisolventHrotfem2d1}, together with
\eqref{eq:tensorvanishedge} and the vanishing of
\eqref{Hincfem13ddof4}--\eqref{Hincfem13ddof5}, implies
\[
\Pi_F((\curl\boldsymbol\tau)^{\intercal})\Pi_F=0
\qquad\text{on each }F\in\Delta_2(T).
\]
Together with \eqref{eq:curltensortnvanishface}, this gives
\begin{equation}\label{eq:curltensorvanishface}
\Pi_F(\curl\boldsymbol\tau)=0
\qquad\text{on each }F\in\Delta_2(T).
\end{equation}

Furthermore,
\[
\rot_F(\Pi_F\boldsymbol\tau\boldsymbol n)
=
\boldsymbol n^{\intercal}(\curl\boldsymbol\tau)^{\intercal}\boldsymbol n
=
-\tr_F(\Pi_F((\curl\boldsymbol\tau)^{\intercal})\Pi_F)
=0.
\]
By the unisolvence of the BDM element in
\cite{BrezziDouglasMarini1985}, \eqref{eq:tensorvanishedge} and
\eqref{Hincfem13ddof9} imply
$\Pi_F\boldsymbol\tau\boldsymbol n=0$
on each $F\in\Delta_2(T)$.
Therefore
\[
\Pi_F\boldsymbol\tau=0
\qquad\text{on each }F\in\Delta_2(T).
\]
This, together with \eqref{eq:curltensorvanishface} and the vanishing vertex
values in \eqref{Hincfem13ddof1}, implies
$\boldsymbol\tau\in \mathbb B_{k+2}^{\inc^+}(T^{\rm R};\mathbb S)$.
Finally, the exact bubble complex \eqref{eq:bubbleelascomplex13d} and the
vanishing degrees of freedom
\eqref{Hincfem13ddof10}--\eqref{Hincfem13ddof11} yield
$\boldsymbol\tau=0$.
\end{proof}

The global $H(\inc^+;\mathbb S)$-conforming finite element space is defined by
\begin{equation}\label{eq:Sigmahinc+k1}
\begin{aligned}
\Sigma_h^{\inc^+}
&:=
\bigl\{
\boldsymbol{\tau}_h\in L^2(\Omega;\mathbb S):
\ \boldsymbol{\tau}_h|_T\in\Sigma_{k+2}^{\inc^+}(T;\mathbb S)
\text{ for all }T\in\mathcal T_h,\\
&\qquad\qquad \text{the degrees of freedom }
\eqref{Hincfem13ddof1}\text{--}\eqref{Hincfem13ddof9}
\text{ are single-valued}
\bigr\}.
\end{aligned}
\end{equation}
By construction,
$\Sigma_h^{\inc^+}
\subset H(\inc^+,\Omega;\mathbb S)$.
Let
$
I_h^{\inc^+}:H^3(\Omega;\mathbb S)\to\Sigma_h^{\inc^+}
$
be the interpolation operator determined by the degrees of freedom
\eqref{Hincfem13ddof}.

\subsection{The finite element elasticity complex}
We now prove the global exactness of the finite element elasticity complex.
\begin{theorem}
Assume that $\Omega$ is contractible.  Then the finite element elasticity complex
\eqref{eq:elascomplex13d} is exact.
\end{theorem}

\begin{proof}
The sequence \eqref{eq:elascomplex13d} is a complex.  

We first prove
\begin{equation}\label{eq:kerincDef}
\Sigma_h^{\inc^+}\cap\ker(\inc)=\defm(V_h^{1,\curl}).
\end{equation}
The inclusion ``$\supseteq$'' is immediate.  Conversely, let
$\boldsymbol{\tau}\in \Sigma_h^{\inc^+}$ satisfy $\inc\boldsymbol{\tau}=0$. 
By the exactness of the continuous complex
\eqref{elascomplex:H1curlinc+} and the local complex
\eqref{eq:localelascomplex13d}, there exists
$\boldsymbol{v}\in
\mathbb P_{k+3}^{-1}(\mathcal{T}_h^{\rm R};\mathbb R^3)
\cap H^1(\curl,\Omega)$ such that
$\boldsymbol{\tau}=\defm(\boldsymbol{v})$.
The data in \eqref{H1femk1dof1} are single-valued by
\eqref{Hincfem13ddof1}, and
$\curl\boldsymbol{v}\in H^1(\Omega;\mathbb R^3)$.  The data in
\eqref{H1femk1dof2} are single-valued because
$\boldsymbol{v}\in H^1(\Omega;\mathbb R^3)$, while those in
\eqref{H1femk1dof3} follow from \eqref{Hincfem13ddof3} by integration by
parts.  Finally, the data in \eqref{H1femk1dof4},
\eqref{H1femk1dof5}--\eqref{H1femk1dof7}, and
\eqref{H1femk1dof8} inherit single-valuedness from
\eqref{Hincfem13ddof5},
\eqref{Hincfem13ddof7}--\eqref{Hincfem13ddof9}, and
\eqref{Hincfem13ddof11}, respectively.
Hence, $\boldsymbol{v}\in V_h^{1,\curl}$ and
$\boldsymbol{\tau}\in \defm(V_h^{1,\curl})$, which proves
\eqref{eq:kerincDef}.

By \eqref{divonto}, 
it remains to show that
\[
\Sigma_{k,h}^{\div}\cap\ker(\div)
=
\inc\Sigma_h^{\inc^+}.
\]

We prove this by a dimension count.  Using the exactness of the local bubble
complex \eqref{eq:bubbleelascomplex13d}, we obtain
\begin{align*}
&\dim(\Sigma_{k,h}^{\div}\cap\ker(\div))
-
\dim\inc\Sigma_h^{\inc^+}
\\
&=
\dim\Sigma_{k,h}^{\div}
-
\dim V_{k-1,h}^{L^2}
-
\dim\Sigma_h^{\inc^+}
+
\dim V_h^{1,\curl}
-
6
\\
&=
9|\Delta_2(\mathcal T_h)|
-
6|\mathcal T_h|
-
\bigl(
6|\Delta_0(\mathcal T_h)|
+
15|\Delta_1(\mathcal T_h)|
+
8|\Delta_2(\mathcal T_h)|
\bigr)
\\
&\quad
+
\bigl(
12|\Delta_0(\mathcal T_h)|
+
9|\Delta_1(\mathcal T_h)|
+
5|\Delta_2(\mathcal T_h)|
\bigr)
-
6
\\
&=
-6|\mathcal T_h|
+
6|\Delta_2(\mathcal T_h)|
-
6|\Delta_1(\mathcal T_h)|
+
6|\Delta_0(\mathcal T_h)|
-
6.
\end{align*}
Euler's formula for a topologically trivial tetrahedral mesh gives
\[
-|\mathcal T_h|
+
|\Delta_2(\mathcal T_h)|
-
|\Delta_1(\mathcal T_h)|
+
|\Delta_0(\mathcal T_h)|
=
1.
\]
Therefore
\[
\dim(\Sigma_{k,h}^{\div}\cap\ker(\div))
=
\dim\inc\Sigma_h^{\inc^+}.
\]
Since $\inc\Sigma_h^{\inc^+}\subseteq\Sigma_{k,h}^{\div}\cap\ker(\div)$,
the two spaces are equal.  This completes the proof.
\end{proof}

The construction also admits the following less regular variant.

\begin{lemma}
Assume that $\Omega$ is contractible.
For $k\geq 1$, the complex
\begin{equation}\label{eq:lowersmoothelascomplex3d}
\begin{aligned}
{\rm RM}
&\xrightarrow{\subset}
V_h^{\rm grad}
\xrightarrow{\defm}
\widetilde{\Sigma}_h^{\inc}
\xrightarrow{\inc}
\Sigma_{k,h}^{\div}
\xrightarrow{\div}
V_{k-1,h}^{L^2}
\to0
\end{aligned}
\end{equation}
is exact, where
\begin{align*}
V_h^{\rm grad}
&:=\{
\boldsymbol{v}_h\in H^1(\Omega;\mathbb R^3):
\boldsymbol{v}_h|_T
\in\mathbb P_{k+3}^{-1}(T^{\rm R};\mathbb R^3)
\text{ for all }T\in\mathcal T_h
\}, \\
\widetilde{\Sigma}_h^{\inc}
&:=\{
\boldsymbol{\tau}_h\in H(\inc,\Omega;\mathbb S):
\boldsymbol{\tau}_h|_T
\in\mathbb P_{k+2}^{-1}(T^{\rm R};\mathbb S)
\text{ for all }T\in\mathcal T_h
\}.
\end{align*}
\end{lemma}
\begin{proof}
It is immediate that
$\widetilde{\Sigma}_h^{\inc}\cap\ker(\inc)=\defm(V_h^{\rm grad})$.
By \eqref{divonto},
$\div\Sigma_{k,h}^{\div}=V_{k-1,h}^{L^2}$.  Moreover, the exactness of
\eqref{eq:elascomplex13d} gives
$\inc\Sigma_h^{\inc^+}=\Sigma_{k,h}^{\div}\cap\ker(\div)$.  Since
\[
\inc\Sigma_h^{\inc^+}
 \subseteq
\inc\widetilde{\Sigma}_h^{\inc}
 \subseteq
\Sigma_{k,h}^{\div}\cap\ker(\div),
\]
we conclude that
$\inc\widetilde{\Sigma}_h^{\inc}=\Sigma_{k,h}^{\div}\cap\ker(\div)$.
\end{proof}

\begin{lemma}
The following commuting property holds:
\begin{equation}\label{eq:commutativityincdiv1}
\inc(I_h^{\inc^+}\boldsymbol{\tau})
=
I_h^{\div}(\inc \boldsymbol{\tau}),
\qquad
\forall\,\boldsymbol{\tau}\in H^3(\Omega;\mathbb S).
\end{equation}
\end{lemma}

\begin{proof}
Set
$
\boldsymbol{\sigma}_h
:=
I_h^{\div}(\inc \boldsymbol{\tau})
-
\inc(I_h^{\inc^+}\boldsymbol{\tau})
\in \Sigma_{k,h}^{\div}$.
It suffices to show that all degrees of freedom
\eqref{HdivSfemdof} vanish for $\boldsymbol{\sigma}_h$.

Let $F\in\Delta_2(T)$ and $q\in\mathbb P_1(F)$.  By the trace identity
\eqref{eq:inctr1},
\[
\begin{aligned}
(\boldsymbol n^{\intercal}\boldsymbol{\sigma}_h\boldsymbol n,q)_F
&=
(\boldsymbol n\cdot
\inc(\boldsymbol\tau-I_h^{\inc^+}\boldsymbol\tau)\cdot
\boldsymbol n,q)_F
\\
&=
(\rot_F\rot_F
\tr_1(\boldsymbol\tau-I_h^{\inc^+}\boldsymbol\tau),q)_F.
\end{aligned}
\]
Applying the Green identity \eqref{eq:greenidentityrotrot2D} and using the
vanishing of the degrees of freedom
\eqref{Hincfem13ddof1}--\eqref{Hincfem13ddof3} for
$\boldsymbol\tau-I_h^{\inc^+}\boldsymbol\tau$, we obtain
\[
(\boldsymbol n^{\intercal}\boldsymbol{\sigma}_h\boldsymbol n,q)_F=0,
\qquad
q\in\mathbb P_1(F).
\]
This, together with the vanishing condition \eqref{Hincfem13ddof6}, yields
$\boldsymbol n^{\intercal}\boldsymbol{\sigma}_h\boldsymbol n=0$ on
$\partial T$.

We next prove
\[
\begin{aligned}
(\Pi_F\boldsymbol{\sigma}_h\boldsymbol n,\boldsymbol q)_F
&=
\bigl(\rot_F\bigl(
\Pi_F(\curl(\boldsymbol{\tau}-I_h^{\inc^+}\boldsymbol{\tau}))^{\intercal}
\Pi_F
\bigr),\boldsymbol q\bigr)_F =0,
\quad
\forall\,\boldsymbol q\in\mathbb P_k(F;\mathbb R^2).
\end{aligned}
\]
By the moment condition \eqref{Hincfem13ddof4}, it suffices to show that
\[
\begin{aligned}
&\bigl(\rot_F\bigl(
\Pi_F(\curl(\boldsymbol{\tau}-I_h^{\inc^+}\boldsymbol{\tau}))^{\intercal}
\Pi_F
\bigr),\boldsymbol q\bigr)_F=0,\quad
\forall\,\boldsymbol q\in{\rm RM}(F).
\end{aligned}
\]
As in the derivation of \eqref{eq:20260607}, this follows by integration by
parts and the vanishing of the degrees of freedom
\eqref{Hincfem13ddof2}--\eqref{Hincfem13ddof3} for
$\boldsymbol\tau-I_h^{\inc^+}\boldsymbol\tau$.
Hence, the face degrees of freedom
\eqref{HdivSfemdof1} of $\boldsymbol{\sigma}_h$ vanish.

By the commuting relation \eqref{eq:commutativitydiv} and $\div\inc=0$,
\[
\div\boldsymbol{\sigma}_h
=
\div(I_h^{\div}(\inc\boldsymbol{\tau}))
=
Q_h(\div\inc\boldsymbol{\tau})
=
0.
\]
Hence the degree of freedom \eqref{HdivSfemdof2} also vanishes. 
Finally, the degree of freedom \eqref{HdivSfemdof3} vanishes by the
definitions of the interpolation operators.  By the unisolvence of the
$H(\div;\mathbb S)$ element,
$\boldsymbol{\sigma}_h=0$, which proves
\eqref{eq:commutativityincdiv1}.
\end{proof}

\begin{lemma}
The following commuting property holds:
\begin{equation}\label{eq:commutativitydefinc1}
\defm(I_h^{1,\curl}\boldsymbol{v})
=
I_h^{\inc^+}(\defm\boldsymbol{v}),
\qquad
\forall\,\boldsymbol{v}\in H^4(\Omega;\mathbb R^3).
\end{equation}
\end{lemma}

\begin{proof}
Set
$
\boldsymbol{\tau}_h
:=
I_h^{\inc^+}(\defm\boldsymbol{v})
-
\defm(I_h^{1,\curl}\boldsymbol{v})
\in \Sigma_h^{\inc^+}$.
We show that all degrees of freedom \eqref{Hincfem13ddof} vanish for
$\boldsymbol{\tau}_h$.

The vertex degrees of freedom \eqref{Hincfem13ddof1} and the degrees of
freedom \eqref{Hincfem13ddof6}--\eqref{Hincfem13ddof11} vanish directly
from \eqref{eq:commutativityincdiv1} and the definitions of
$I_h^{1,\curl}$ and $I_h^{\inc^+}$.  The edge degrees of freedom
\eqref{Hincfem13ddof2}--\eqref{Hincfem13ddof3} vanish by integration by
parts on each edge and by the corresponding degrees of freedom of
$I_h^{1,\curl}$.

Moreover, using the identity
\[
2\Pi_F\bigl((\curl\defm\boldsymbol v)^{\intercal}\bigr)\Pi_F
=
\grad_F\bigl(\Pi_F(\curl\boldsymbol v)\bigr),
\]
the face degrees of freedom
\eqref{Hincfem13ddof4}--\eqref{Hincfem13ddof5} also vanish by the definition
of $I_h^{1,\curl}$.  Thus all degrees of freedom of $\boldsymbol\tau_h$
vanish.  By Lemma~\ref{lem:unisolHincfem13ddof}, $\boldsymbol\tau_h=0$,
which proves \eqref{eq:commutativitydefinc1}.
\end{proof}

\begin{theorem}
The interpolation operators form the commuting diagram
\[
\begin{array}{c}
\xymatrix{
{\rm RM}\ar[r]
&
H^4(\Omega;\mathbb R^3)
\ar[r]^-{\defm}
\ar[d]^{I_h^{1,\curl}}
&
H^3(\Omega;\mathbb S)
\ar[r]^-{\inc}
\ar[d]^{I_h^{\inc^+}}
&
H^1(\Omega;\mathbb S)
\ar[r]^-{\div}
\ar[d]^{I_h^{\div}}
&
L^2(\Omega;\mathbb R^3)
\ar[r]
\ar[d]^{Q_h}
&
0
\\
{\rm RM}\ar[r]
&
V_h^{1,\curl}
\ar[r]^-{\defm}
&
\Sigma_h^{\inc^+}
\ar[r]^-{\inc}
&
\Sigma_{k,h}^{\div}
\ar[r]^-{\div}
&
V_{k-1,h}^{L^2}
\ar[r]
&
0.
}
\end{array}
\]
\end{theorem}
\begin{proof}
The commutative diagram is obtained by combining \eqref{eq:commutativitydiv},
\eqref{eq:commutativityincdiv1}, and
\eqref{eq:commutativitydefinc1}. 
\end{proof}

\section{Finite element complex for the $H^1$--$H(\inc)$ elasticity sequence}\label{sec:femelascomplexinc}

This section constructs a finite element elasticity complex for the
$H^1$--$H(\inc)$ elasticity sequence on the Alfeld refinement of a
tetrahedral mesh.  Throughout this section, we assume $k\geq2$.  The discrete
complex is
\begin{equation}\label{eq:elascomplex3d}
{\rm RM}
\xrightarrow{\subset}
V_h^{\rm herm}
\xrightarrow{\defm}
\Sigma_h^{\inc}
\xrightarrow{\inc}
\Sigma_{k,h}^{\div}
\xrightarrow{\div}
V_{k-1,h}^{L^2}
\to0.
\end{equation}
Here $V_h^{\rm herm}$, $\Sigma_h^{\inc}$, and
$\Sigma_{k,h}^{\div}$ are defined in
\eqref{eq:Vhherm}, \eqref{eq:Sigmahinc}, and
\eqref{eq:Sigmahdiv}, respectively.  The sequence
\eqref{eq:elascomplex3d} is a finite element subcomplex of the continuous
elasticity complex \eqref{elascomplex3d}.
We also construct commuting interpolation operators for
\eqref{eq:elascomplex3d}.

\subsection{Finite elements for symmetric tensors on faces}
The $H(\inc)$ trace of a symmetric tensor has two components.  On each face
$F$, the trace $\tr_1(\boldsymbol\tau)$ is governed by the scalar operator
$\rot_F\rot_F$, whereas $\tr_2(\boldsymbol\tau)$ is governed by $\rot_F$; see
\eqref{eq:inctr1}--\eqref{eq:inctr2}.  We therefore begin with the two face
finite elements that will be used as trace elements in the three-dimensional
construction.

Let $F$ be a triangular face and identify tangential symmetric tensors on $F$
with two-dimensional symmetric matrices in a fixed tangential frame.  For
$k\geq2$, the first trace element has shape space
$\mathbb P_{k+2}(F;\mathbb S_F)$ and is conforming for
$H(\rot_F\rot_F,F;\mathbb S_F)$.  Its degrees of freedom are
\begin{subequations}\label{Hrotrotfem2ddof}
\begin{align}
\boldsymbol{\tau}(\delta), \nabla_F\boldsymbol{\tau}(\delta), & \quad\delta\in \Delta_0(F), \label{Hrotrotfem2ddof1}\\
(\boldsymbol{\tau}, \boldsymbol  q)_e, & \quad\boldsymbol  q\in\mathbb P_{k-2}(e;\mathbb S_F),  e\in\Delta_1(F),\label{Hrotrotfem2ddof2}\\
(\boldsymbol t^{\intercal}\rot_F\boldsymbol{\tau}, q)_e, & \quad q\in\mathbb P_{k-1}(e),  e\in\Delta_1(F),\label{Hrotrotfem2ddof3}\\
(\rot_F\rot_F\boldsymbol{\tau}, q)_F, & \quad q\in\mathbb P_{k}(F)/\mathbb P_{1}(F),\label{Hrotrotfem2ddof4}\\
(\boldsymbol  \tau, \boldsymbol  q)_F, & \quad\boldsymbol  q\in\sym(\boldsymbol x\otimes\mathbb P_{k-3}(F;\mathbb R^2)). \label{Hrotrotfem2ddof5}
\end{align}
\end{subequations}

\begin{lemma}\label{lem:unisolventHrotrotfem2d}
The degrees of freedom \eqref{Hrotrotfem2ddof} are unisolvent for
$\mathbb P_{k+2}(F;\mathbb S_F)$.
\end{lemma}
\begin{proof}
The number of degrees of freedom \eqref{Hrotrotfem2ddof} is
\begin{equation*}
\begin{aligned}
&27 + 9(k-1) + 3k
 + \frac{1}{2}(k+1)(k+2) - 3 + (k-1)(k-2) = \frac{3}{2}(k+3)(k+4),
\end{aligned}
\end{equation*}
which equals $\dim\mathbb P_{k+2}(F;\mathbb S_F)$.  It remains to show uniqueness.

Let $\boldsymbol{\tau}\in \mathbb P_{k+2}(F;\mathbb S_F)$ and suppose that all
degrees of freedom \eqref{Hrotrotfem2ddof} vanish.  The vanishing of
\eqref{Hrotrotfem2ddof1}--\eqref{Hrotrotfem2ddof3} implies that
$\boldsymbol{\tau}$ and $\boldsymbol t^{\intercal}\rot_F\boldsymbol{\tau}$
vanish on $\partial F$.
Hence the Green identity \eqref{eq:greenidentityrotrot2D} gives
\[
(\rot_F\rot_F\boldsymbol\tau,q)_F=0,
\qquad q\in\mathbb P_1(F).
\]
Together with \eqref{Hrotrotfem2ddof4}, this yields
$\rot_F\rot_F\boldsymbol\tau=0$ on $F$.

By the two-dimensional elasticity complex there is
$\boldsymbol v\in\mathbb P_{k+3}(F;\mathbb R^2)$, unique after fixing a rigid
motion, such that $\boldsymbol\tau=\defm_F\boldsymbol v$.  We fix the rigid
motion by requiring $\boldsymbol v$ and $\rot_F\boldsymbol v$ to vanish at one
vertex.  Since
\[
\partial_t(\rot_F\boldsymbol v)
=
2\boldsymbol t^{\intercal}\rot_F\boldsymbol\tau
=0
\quad\text{on }\partial F,
\]
we have $\rot_F\boldsymbol v=0$ on $\partial F$.  Combining this with
$\defm_F\boldsymbol v=0$ on $\partial F$ gives
$\nabla_F\boldsymbol v=0$ on $\partial F$, and therefore
$\boldsymbol v\in\mathbb P_{k+3}(F;\mathbb R^2)\cap H_0^2(F;\mathbb R^2)$.
Equivalently, 
$\boldsymbol v=b_F^2\boldsymbol p$ with
$\boldsymbol p\in\mathbb P_{k-3}(F;\mathbb R^2)$.
Finally, integration by parts and the vanishing of
\eqref{Hrotrotfem2ddof5} yield $\boldsymbol{v}=0$, and hence
$\boldsymbol{\tau}=0$.
\end{proof}

This $H(\rot_F\rot_F,F;\mathbb S_F)$ element is the rotated form of the
$H(\div_F\div_F,F;\mathbb S_F)$ element in
\cite[(5.10)]{ChenHuang2025a}, with smoothness vectors
$\boldsymbol{r}_1=(1,0)^{\intercal}$ and $\boldsymbol{r}_2=-1$.

The second face element controls $\tr_2$.  It is an
$H(\rot_F,F;\mathbb S_F)$ element with shape space
$\mathbb P_{k+1}(F;\mathbb S_F)$ and degrees of freedom
\begin{subequations}\label{Hrotfem2ddof}
\begin{align}
\boldsymbol{\tau}(\delta), & \quad\delta\in \Delta_0(F), \label{Hrotfem2ddof1}\\
(\boldsymbol{\tau}\boldsymbol t, \boldsymbol q)_e, & \quad\boldsymbol  q\in\mathbb P_{k-1}(e; \mathbb R^2),  e\in\Delta_1(F),\label{Hrotfem2ddof2}\\
(\rot_F\boldsymbol  \tau, \boldsymbol  q)_F, & \quad\boldsymbol  q\in\mathbb P_{k}(F; \mathbb R^2)/{\rm RT}(F),\label{Hrotfem2ddof3}\\
(\boldsymbol  \tau, \boldsymbol  q)_F, & \quad\boldsymbol  q\in\boldsymbol  x\boldsymbol  x^{\intercal}\mathbb P_{k-3}(F). \label{Hrotfem2ddof4}
\end{align}
\end{subequations}
This is the rotated form of the two-dimensional Hu--Zhang element; see
\cite{Hu2015} and \cite[Theorem~4.13]{ChenHuang2022b}.  In particular, the
functionals \eqref{Hrotfem2ddof} are unisolvent for
$\mathbb P_{k+1}(F;\mathbb S_F)$.

\subsection{Hermite-type vector finite elements}
For $k\geq2$, we take $V_{k+3}^{\hess}(T^{\rm R})$ as the scalar local
shape function space.  Its degrees of freedom are
(cf. \cite[(4.11)]{FuGuzmanNeilan2020})
\begin{subequations}\label{H1femdof}
\begin{align}
v (\delta), \nabla v (\delta), \nabla^2 v (\delta), & \quad \delta\in \Delta_0(T), \label{H1femdof1}\\
(v, q)_e, & \quad q\in\mathbb P_{k-3}(e),  e\in\Delta_1(T),\label{H1femdof2}\\
(\partial_{n_i}v, q)_e, & \quad q\in\mathbb P_{k-2}(e),  e\in\Delta_1(T), i=1,2,\label{H1femdof3}\\
(v, q)_F, & \quad q\in\mathbb P_{k-3}(F),  F\in\Delta_2(T),\label{H1femdof4}\\
(v, q)_T, & \quad q\in \mathbb B_{k+3}^{\rm herm}(T^{\rm R}). \label{H1femdof5}
\end{align}
\end{subequations}
\begin{lemma}
The degrees of freedom \eqref{H1femdof} are unisolvent for
$V_{k+3}^{\hess}(T^{\rm R})$.
\end{lemma}

\begin{proof}
By \eqref{eq:bubblehermdim},
the number of degrees of freedom is
\[
\begin{aligned}
&40+6(3k-4)+2(k-1)(k-2)+\frac{2}{3}k(k+1)(k+2) =
\frac{2}{3}(k^3+6k^2+20k+30),
\end{aligned}
\]
which equals $\dim V_{k+3}^{\hess}(T^{\rm R})$ by
\eqref{eq:Vhessdim}.  Hence it remains to prove uniqueness.

Let $v\in V_{k+3}^{\hess}(T^{\rm R})$ and suppose that all degrees
of freedom in \eqref{H1femdof} vanish.  From the vanishing of
\eqref{H1femdof1}--\eqref{H1femdof4}, we obtain
$v\in \mathbb B_{k+3}^{\rm herm}(T^{\rm R})$.
The vanishing of \eqref{H1femdof5} then yields $v=0$.
\end{proof}

Define the global vector finite element space by
\begin{equation}\label{eq:Vhherm}
\begin{aligned}
V_h^{\rm herm}
:=\{\boldsymbol{v}_h\in H^1(\Omega;\mathbb R^3)&:
\boldsymbol{v}_h|_T\in
V_{k+3}^{\hess}(T^{\rm R};\mathbb R^3)
\text{ for all } T\in\mathcal T_h, \\
&\quad
\text{the degrees of freedom }
\eqref{H1femdof1}\text{--}\eqref{H1femdof4}
\text{ are single-valued}
\}.
\end{aligned}
\end{equation}

Let $I_h^{\rm herm}:H^4(\Omega;\mathbb R^3)\to V_h^{\rm herm}$ denote the
interpolation operator defined by the degrees of freedom \eqref{H1femdof}.

\subsection{$H(\inc;\mathbb S)$-conforming finite elements}
We now define the middle tensor space.  The local shape space is
$\Sigma_{k+2}^{1,\inc}(T;\mathbb S)$ and the degrees of freedom are
\begin{subequations}\label{Hincfem3ddof}
\begin{align}
\boldsymbol\tau (\delta), \nabla\boldsymbol\tau (\delta),
& \quad
\delta\in \Delta_0(T),
\label{Hincfem3ddof1}\\
(\boldsymbol\tau, \boldsymbol  q)_e,
& \quad
\boldsymbol  q\in\mathbb P_{k-2}(e;\mathbb S),\;
e\in\Delta_1(T),
\label{Hincfem3ddof2}\\
((\curl\boldsymbol\tau)^{\intercal}\boldsymbol t, \boldsymbol q)_e,
& \quad
\boldsymbol  q\in\mathbb P_{k-1}(e;\mathbb R^3),\;
e\in\Delta_1(T),
\label{Hincfem3ddof3}\\
(\rot_F\rot_F\tr_1(\boldsymbol  \tau), q)_F,
& \quad
q\in\mathbb P_{k}(F)/\mathbb P_{1}(F),\;
F\in\Delta_2(T),
\label{Hincfem3ddof4}\\
(\tr_1(\boldsymbol  \tau), \boldsymbol  q)_F,
& \quad
\boldsymbol  q\in
\sym(\boldsymbol x\otimes\mathbb P_{k-3}(F;\mathbb R^2)),\;
F\in\Delta_2(T),
\label{Hincfem3ddof5}\\
(\rot_F\tr_2(\boldsymbol  \tau), \boldsymbol  q)_F,
& \quad
\boldsymbol  q\in\mathbb P_{k}(F; \mathbb R^2)/{\rm RT}(F),\;
F\in\Delta_2(T),
\label{Hincfem3ddof6}\\
(\tr_2(\boldsymbol  \tau), \boldsymbol  q)_F,
& \quad
\boldsymbol  q\in
\boldsymbol  x\boldsymbol  x^{\intercal}\mathbb P_{k-3}(F),\;
F\in\Delta_2(T),
\label{Hincfem3ddof7}\\
(\inc\boldsymbol\tau, \boldsymbol q)_T,
& \quad
\boldsymbol q\in
\mathbb B_k^{\div}(T^{\rm R};\mathbb S)\cap\ker(\div),
\label{Hincfem3ddof8}\\
(\boldsymbol\tau, \boldsymbol q)_T,
& \quad
\boldsymbol q\in
\defm\bigl(\mathbb B_{k+3}^{\rm herm}(T^{\rm R};\mathbb R^3)\bigr).
\label{Hincfem3ddof9}
\end{align}
\end{subequations}

\begin{lemma}
The degrees of freedom \eqref{Hincfem3ddof} are unisolvent for
$\Sigma_{k+2}^{1,\inc}(T;\mathbb S)$.
\end{lemma}
\begin{proof}
We first count the degrees of freedom.
The number of degrees of freedom in
\eqref{Hincfem3ddof1}--\eqref{Hincfem3ddof7} is
\begin{align*}
 96+6(9k-6)+4\biggl[3\binom{k+2}{2}+3\binom{k-1}{2}-6\biggr]
 =12k^2+54k+60.
\end{align*}
By \eqref{eq:BincSdim} and the exact bubble elasticity complex
\eqref{eq:bubbleelascomplex3d}, the number of degrees of freedom in
\eqref{Hincfem3ddof8}--\eqref{Hincfem3ddof9} is
$4k^3+9k^2-k$.
Hence the total number of degrees of freedom in \eqref{Hincfem3ddof} equals
the dimension of $\Sigma_{k+2}^{1,\inc}(T;\mathbb S)$ given in
\eqref{eq:Sigmma1incSdim}.

Assume $\boldsymbol\tau\in\Sigma_{k+2}^{1,\inc}(T;\mathbb S)$ and that all
degrees of freedom in \eqref{Hincfem3ddof} vanish.  The vanishing of
\eqref{Hincfem3ddof1}--\eqref{Hincfem3ddof3} implies that
$\boldsymbol\tau|_e=0$ and
$((\curl\boldsymbol\tau)^{\intercal}\boldsymbol t)|_e=0$ for each edge
$e\in\Delta_1(T)$.
Lemma~\ref{lem:unisolventHrotrotfem2d}, the unisolvence of the degrees of
freedom
\eqref{Hrotfem2ddof}, the identities
\eqref{eq:trrotFrotFtr1inc}--\eqref{eq:trrotFtr2inc}, and the vanishing of
the degrees of freedom
\eqref{Hincfem3ddof4}--\eqref{Hincfem3ddof7} yield
\[
\tr_1(\boldsymbol\tau)=0,
\qquad
\tr_2(\boldsymbol\tau)=0
\qquad\text{on }\partial T.
\]
Consequently \eqref{eq:inctr1}--\eqref{eq:inctr2} imply
$\inc\boldsymbol\tau\in\mathbb B_k^{\div}(T^{\rm R};\mathbb S)$.  The
vanishing of \eqref{Hincfem3ddof8} then gives
$\inc\boldsymbol\tau=0$.

By the local exactness \eqref{eq:localelascomplex3d},
$\boldsymbol{\tau}=\defm(\boldsymbol{v})$ for some
$\boldsymbol{v}\in V_{k+3}^{\hess}(T^{\rm R};\mathbb R^3)$ satisfying
$\boldsymbol{v}|_e=0$ and $(\nabla\boldsymbol{v})|_e=0$ for all
$e\in\Delta_1(T)$.  The trace identities \eqref{eq:trdef}, together with
$\tr_1(\boldsymbol\tau)=\tr_2(\boldsymbol\tau)=0$, show that
$\boldsymbol v\in\mathbb B_{k+3}^{\rm herm}(T^{\rm R};\mathbb R^3)$.
Finally, $\boldsymbol v=0$ follows from the vanishing of \eqref{Hincfem3ddof9}.
\end{proof}

The global $H(\inc;\mathbb S)$-conforming finite element space is
\begin{equation}\label{eq:Sigmahinc}
\begin{aligned}
\Sigma_h^{\inc}
:=\{\boldsymbol{\tau}_h\in L^2(\Omega;\mathbb S)&:
\boldsymbol{\tau}_h|_T\in\Sigma_{k+2}^{1,\inc}(T;\mathbb S)
\text{ for all }T\in\mathcal T_h, \\
&\quad
\text{the degrees of freedom }
\eqref{Hincfem3ddof1}\text{--}\eqref{Hincfem3ddof7}
\text{ are single-valued}
\}.
\end{aligned}
\end{equation}
Lemma~\ref{lem:unisolventHrotrotfem2d}, the unisolvence of the degrees of
freedom
\eqref{Hrotfem2ddof}, and the Green identity
\eqref{eq:greenidentityrotrot2D} show that
$\Sigma_h^{\inc}\subset H(\inc,\Omega;\mathbb S)$.

Let $I_h^{\inc}:H^3(\Omega;\mathbb S)\to\Sigma_h^{\inc}$ denote the
interpolation operator defined by the degrees of freedom
\eqref{Hincfem3ddof}.

\begin{lemma}
For $k\geq2$, we have
\begin{equation}\label{eq:commutativityincdiv}
\inc(I_h^{\inc}\boldsymbol{\tau})
=
I_h^{\div}(\inc \boldsymbol{\tau}),
\qquad
\forall\,\boldsymbol{\tau}\in H^3(\Omega;\mathbb S).
\end{equation}
\end{lemma}
\begin{proof}
Set
$
\boldsymbol\sigma_h
:=
I_h^{\div}(\inc\boldsymbol\tau)
-
\inc(I_h^{\inc}\boldsymbol\tau)
\in\Sigma_{k,h}^{\div}$.  We show that all degrees of freedom
\eqref{HdivSfemdof} of $\boldsymbol\sigma_h$ vanish.

By \eqref{eq:inctr1}, for any $F\in\Delta_2(T)$ and $q\in\mathbb P_1(F)$,
\[
\begin{aligned}
(\boldsymbol n\cdot\boldsymbol{\sigma}_h\cdot \boldsymbol n, q)_F
&=
(\boldsymbol n\cdot
\inc(\boldsymbol\tau-I_h^{\inc}\boldsymbol\tau)
\cdot \boldsymbol n, q)_F =
(\rot_F\rot_F\tr_1(\boldsymbol\tau-I_h^{\inc}\boldsymbol\tau), q)_F.
\end{aligned}
\]
Applying the Green identity \eqref{eq:greenidentityrotrot2D} and
\eqref{eq:trrotFrotFtr1inc}, and using the vanishing degrees of freedom
\eqref{Hincfem3ddof1}--\eqref{Hincfem3ddof3} of
$\boldsymbol\tau-I_h^{\inc}\boldsymbol\tau$, we obtain
\[
(\boldsymbol n\cdot\boldsymbol{\sigma}_h\cdot \boldsymbol n, q)_F=0,
\qquad
\forall\,q\in\mathbb P_1(F).
\]
Together with the vanishing of \eqref{Hincfem3ddof4}, this yields
\[
(\boldsymbol n\cdot\boldsymbol{\sigma}_h\cdot \boldsymbol n, q)_F=0,
\qquad
\forall\,q\in\mathbb P_k(F).
\]
Similarly, using \eqref{eq:inctr2}, integration by parts,
\eqref{eq:trrotFtr2inc}, and the degrees of freedom
\eqref{Hincfem3ddof1}--\eqref{Hincfem3ddof3} and
\eqref{Hincfem3ddof6}, we have
\[
(\boldsymbol n\times\boldsymbol{\sigma}_h\cdot \boldsymbol n,
\boldsymbol q)_F=0,
\qquad
\forall\,\boldsymbol q\in\mathbb P_k(F;\mathbb R^2).
\]
The last two equations show that the face degrees of freedom
\eqref{HdivSfemdof1} of $\boldsymbol{\sigma}_h$ vanish.

Next, by the commutativity \eqref{eq:commutativitydiv} and $\div\inc=0$,
\[
\div\boldsymbol{\sigma}_h
=
\div(I_h^{\div}(\inc \boldsymbol{\tau}))
=
Q_h(\div(\inc \boldsymbol{\tau}))
=0.
\]
Hence the degrees of freedom \eqref{HdivSfemdof2} of
$\boldsymbol{\sigma}_h$ also vanish.  Finally, the vanishing of
\eqref{HdivSfemdof3} follows directly from the definitions of
$I_h^{\div}$ and $I_h^{\inc}$.

Thus all degrees of freedom in \eqref{HdivSfemdof} vanish for
$\boldsymbol{\sigma}_h$, and hence $\boldsymbol{\sigma}_h=0$.
\end{proof}

\subsection{Finite element elasticity complex}

\begin{lemma}
Assume that $\Omega$ is contractible.
For $k\geq 2$, the complex \eqref{eq:elascomplex3d} is exact.
\end{lemma}
\begin{proof}
It is immediate that \eqref{eq:elascomplex3d} is a complex. 
First, $\div\Sigma_{k,h}^{\div}=V_{k-1,h}^{L^2}$ is given in \eqref{divonto}.

We then prove
\begin{equation*}
\Sigma_h^{\inc}\cap\ker(\inc)=\defm(V_h^{\rm herm}).    
\end{equation*}
The inclusion ``$\supseteq$'' is immediate.  Conversely, let
$\boldsymbol{\tau}\in\Sigma_h^{\inc}$ satisfy
$\inc\boldsymbol{\tau}=0$.
By the exactness of the continuous complex \eqref{elascomplex3d} and the
local complex \eqref{eq:localelascomplex3d}, there exists
$\boldsymbol{v}\in H^1(\Omega;\mathbb R^3)$ such that
$\boldsymbol{\tau}=\defm(\boldsymbol{v})$ and
$\boldsymbol{v}|_T\in V_{k+3}^{\hess}(T^{\rm R};\mathbb R^3)$ for each
$T\in\mathcal{T}_h$.
The single-valuedness of $\boldsymbol{v}$ at vertices and of the degrees of
freedom
\eqref{H1femdof2} and \eqref{H1femdof4} follows from
$\boldsymbol{v}\in H^1(\Omega;\mathbb R^3)$.
Using the identity (cf. \cite[Lemma~6.4]{ChenHuang2022})
\[
\partial_{ij} v_k
=
\partial_i(\defm \boldsymbol v)_{jk}
+
\partial_j(\defm \boldsymbol v)_{ki}
-
\partial_k(\defm \boldsymbol v)_{ij},
\]
the single-valuedness of the $\nabla^2\boldsymbol{v}$ data in
\eqref{H1femdof1} follows from that of \eqref{Hincfem3ddof1}.
Since
\begin{equation*}
\begin{aligned}
\partial_n(\Pi_F\boldsymbol{v})
&=
2\Pi_F\boldsymbol{\tau}\boldsymbol{n}
-
\nabla_F(\boldsymbol{v}\cdot\boldsymbol{n}),\quad
\partial_n(\boldsymbol{v}\cdot\boldsymbol{n})
=
\boldsymbol{n}^{\intercal}\boldsymbol{\tau}\boldsymbol{n},
\end{aligned}
\end{equation*}
the degrees of freedom
\eqref{Hincfem3ddof1}--\eqref{Hincfem3ddof2} imply that
$(\partial_n\boldsymbol{v})|_e$ is continuous across $F$ for every
$e\in\Delta_1(F)$.  Hence $(\grad\boldsymbol{v})|_e$ is continuous across
$F$ and is therefore single-valued on each edge.  It follows that the
$\nabla\boldsymbol{v}$ data in \eqref{H1femdof1} and
\eqref{H1femdof3} are single-valued.
Thus $\boldsymbol{v}\in V_h^{\rm herm}$, and therefore
$\boldsymbol{\tau}\in \defm(V_h^{\rm herm})$.

We finally prove
$\Sigma_{k,h}^{\div}\cap\ker(\div)=\inc\Sigma_h^{\inc}$
by a dimension count.
By the exactness of the bubble complex \eqref{eq:bubbleelascomplex3d},
\begin{align*}
&\dim(\Sigma_{k,h}^{\div}\cap\ker(\div))
-
\dim\inc\Sigma_h^{\inc} \\
&=
\dim\Sigma_{k,h}^{\div}
-
\dim V_{k-1,h}^{L^2}
-
\dim\Sigma_h^{\inc}
+
\dim V_h^{\rm herm}
-
6 \\
&=
-6|\mathcal T_h|
+
6|\Delta_2(\mathcal T_h)|
-
6|\Delta_1(\mathcal T_h)|
+
6|\Delta_0(\mathcal T_h)|
-
6.
\end{align*}
Euler's formula
\[
-|\mathcal T_h|
+
|\Delta_2(\mathcal T_h)|
-
|\Delta_1(\mathcal T_h)|
+
|\Delta_0(\mathcal T_h)|
=1
\]
gives
$\dim(\Sigma_{k,h}^{\div}\cap\ker(\div))=\dim\inc\Sigma_h^{\inc}$.
Since \eqref{eq:elascomplex3d} is a complex,
$\inc\Sigma_h^{\inc}\subseteq\Sigma_{k,h}^{\div}\cap\ker(\div)$, and
the equality of dimensions proves equality of the two spaces.
\end{proof}

\begin{lemma}
For $k\geq2$, the following commuting property holds:
\begin{equation}\label{eq:commutativitydefinc}
\defm(I_h^{\rm herm}\boldsymbol{v})
=
I_h^{\inc}(\defm\boldsymbol{v}),
\qquad
\forall\,\boldsymbol{v}\in H^4(\Omega;\mathbb R^3).
\end{equation}
\end{lemma}
\begin{proof}
Set
$
\boldsymbol{\tau}_h
:=
I_h^{\inc}(\defm\boldsymbol{v})
-
\defm(I_h^{\rm herm}\boldsymbol{v})
\in \Sigma_h^{\inc}$.
It suffices to show that all degrees of freedom in \eqref{Hincfem3ddof}
vanish for $\boldsymbol{\tau}_h$.

The degree of freedom \eqref{Hincfem3ddof1} vanishes by the definitions of
$I_h^{\rm herm}$ and $I_h^{\inc}$.  The degrees of freedom
\eqref{Hincfem3ddof2}--\eqref{Hincfem3ddof3} vanish by integration by
parts on each edge.  By the trace identity \eqref{eq:trdef}, the degrees of
freedom \eqref{Hincfem3ddof4} and \eqref{Hincfem3ddof6} also vanish.

By \eqref{eq:trdef}, for any
$\boldsymbol q\in\sym(\boldsymbol x\otimes\mathbb P_{k-3}(F;\mathbb R^2))$,
\begin{align*}
(\tr_1(\boldsymbol  \tau_h), \boldsymbol  q)_F
&=
(\tr_1(\defm(\boldsymbol{v}-I_h^{\rm herm}\boldsymbol{v})),
\boldsymbol  q)_F =
(\sym\grad_F(\boldsymbol{v}-I_h^{\rm herm}\boldsymbol{v}),
\boldsymbol  q)_F \\
&=
-(\boldsymbol{v}-I_h^{\rm herm}\boldsymbol{v},
\div_F\boldsymbol  q)_F
=0,
\end{align*}
and for any
$\boldsymbol q\in\boldsymbol  x\boldsymbol  x^{\intercal}\mathbb P_{k-3}(F)$,
\begin{align*}
(\tr_2(\boldsymbol  \tau_h), \boldsymbol  q)_F
&=
(\tr_2(\defm(\boldsymbol{v}-I_h^{\rm herm}\boldsymbol{v})),
\boldsymbol  q)_F =
(\nabla_F^2((\boldsymbol{v}-I_h^{\rm herm}\boldsymbol{v})
\cdot\boldsymbol{n}), \boldsymbol  q)_F \\
&=
(\boldsymbol{v}-I_h^{\rm herm}\boldsymbol{v},
\div_F\div_F\boldsymbol  q)_F
=0.
\end{align*}
Hence the degrees of freedom \eqref{Hincfem3ddof5} and
\eqref{Hincfem3ddof7} vanish.
Therefore, $\boldsymbol \tau_h\in\mathbb B_{k+2}^{\inc}(T^{\rm R};\mathbb S)$.

The degree of freedom \eqref{Hincfem3ddof8} vanishes because
$\inc(\defm(\boldsymbol{v}-I_h^{\rm herm}\boldsymbol{v}))=0$.
By the exactness of the bubble complex \eqref{eq:bubbleelascomplex3d},
$\boldsymbol \tau_h|_T
\in
\defm\bigl(\mathbb B_{k+3}^{\rm herm}(T^{\rm R};\mathbb R^3)\bigr)$ for
$T\in\mathcal{T}_h$.
This, together with the vanishing degree of freedom \eqref{Hincfem3ddof9},
implies $\boldsymbol \tau_h=0$.
\end{proof}

Combining the commuting properties
\eqref{eq:commutativitydiv}, \eqref{eq:commutativityincdiv}, and
\eqref{eq:commutativitydefinc}, we obtain the following commuting
diagram.
\begin{equation*}
\begin{array}{c}
\xymatrix{
{\rm RM}\ar[r]^-{} &H^4(\Omega;\mathbb R^3)\ar[r]^-{\defm} \ar[d]^{I^{\rm herm}_h}&  H^3(\Omega;\mathbb S) \ar[r]^-{\inc} \ar[d]^{I^{\inc}_h}&  H^1(\Omega;\mathbb S) \ar[r]^-{\div} \ar[d]^{I^{\div}_h}&  L^{2}(\Omega;\mathbb R^3) \ar[r]^-{} \ar[d]^{Q_h}& 0\\
{\rm RM}\ar[r]^-{} & V_h^{\rm herm}\ar[r]^-{\defm} & \Sigma_h^{\inc}\ar[r]^-{\inc} & \Sigma_{k,h}^{\div}\ar[r]^-{\div } & V_{k-1,h}^{L^2} \ar[r]^-{} & 0.
}
\end{array}
\end{equation*}

\appendix
\section{Proofs of the bubble exactness results}\label{app:bubble-complexes}

This appendix proves the bubble exactness and dimension results collected in
Subsection~\ref{subsec:bubble-results}.  The argument applies the local BGG
construction to suitable bubble de Rham complexes.

\subsection{The symmetric divergence bubbles and the smoother complex}

Introduce the matrix-valued bubble spaces
\begin{align*}
\mathbb B_{k+2}^{\hess}(T^{\rm R};\mathbb R^3)
&:=V_{k+2}^{\hess}(T^{\rm R};\mathbb R^3)\cap H_0^2(T;\mathbb R^3),\\
\mathbb B_{k+2}^{1,\curl}(T^{\rm R})
&:=V_{k+2}^{1,\curl}(T^{\rm R})\cap H_0^1(\curl,T),\\
\mathbb B_{k+2}^{1,\curl}(T^{\rm R};\mathbb M)
&:=\mathbb R^3\otimes\mathbb B_{k+2}^{1,\curl}(T^{\rm R}),\\
\mathbb B_{k+1}^{\grad}(T^{\rm R};\mathbb M)
&:=\mathbb P_{k+1}^{\grad}(T^{\rm R};\mathbb M)\cap H_0^1(T;\mathbb M),\\
\mathbb B_{k+2}^{\grad}(T^{\rm R};\mathbb X)
&:=\mathbb P_{k+2}^{\grad}(T^{\rm R};\mathbb X)\cap H_0^1(T;\mathbb X),
\qquad \mathbb X\in\{\mathbb R^3,\mathbb M\},\\
\mathbb B_k^{\div}(T^{\rm R};\mathbb M)
&:=\{\boldsymbol\tau\in\Sigma_k^{\div}(T;\mathbb M)
\cap H_0(\div,T;\mathbb M):
\int_T\vskw\boldsymbol\tau\,\dx=0\}.
\end{align*}
Functions in $\mathbb B_{k+2}^{1,\curl}(T^{\rm R})$ have vanishing
first-order derivatives at the vertices of $T$.  The relevant bubble de Rham
complexes form the diagram
\begin{equation}\label{eq:localbggbubblediagram_1}
\begin{tikzcd}[column sep=small]
\mathbb B_{k+3}^{\hess}(T^{\rm R};\mathbb R^3)
 \arrow{r}{\grad}
&
\mathbb B_{k+2}^{1,\curl}(T^{\rm R};\mathbb M)
 \arrow{r}{\curl}
&
\mathbb B_{k+1}^{\grad}(T^{\rm R};\mathbb M)
 \arrow{r}{\div}
&
\mathbb P_k^{-1}(T^{\rm R};\mathbb R^3)/\mathbb R^3
 \to0
\\
\mathbb B_{k+2}^{\hess}(T^{\rm R};\mathbb R^3)
 \arrow[ur,swap,"\mskw"']
 \arrow{r}{\grad}
&
\mathbb B_{k+1}^{\grad}(T^{\rm R};\mathbb M)
 \arrow[ur,swap,"S"']
 \arrow{r}{\curl}
&
\mathbb B_k^{\div}(T^{\rm R};\mathbb M)
 \arrow[ur,swap,"-2\vskw"']
 \arrow{r}{\div}
&
\mathbb P_{k-1}^{-1}(T^{\rm R};\mathbb R^3)/{\rm RM}
 \to0.
\end{tikzcd}
\end{equation}

\begin{lemma}\label{lem:appendix-bubble-derham-smooth}
For $k\geq1$, both rows in \eqref{eq:localbggbubblediagram_1} are exact.
\end{lemma}
\begin{proof}
By \cite[Theorem~3.1 and Corollary~3.4]{FuGuzmanNeilan2020}, it remains only
to prove
\begin{equation}\label{divontohatB}
\div\mathbb B_k^{\div}(T^{\rm R};\mathbb M)
=\mathbb P_{k-1}^{-1}(T^{\rm R};\mathbb R^3)/{\rm RM}.
\end{equation}
The forward inclusion is immediate.  Conversely, let
$\boldsymbol v\in\mathbb P_{k-1}^{-1}(T^{\rm R};\mathbb R^3)/{\rm RM}$.
Since $\boldsymbol v$ is also orthogonal to $\mathbb R^3$, there exists
$\boldsymbol\tau\in\Sigma_k^{\div}(T;\mathbb M)\cap
H_0(\div,T;\mathbb M)$ such that $\div\boldsymbol\tau=\boldsymbol v$.
For a constant skew matrix $\boldsymbol K$, take the rigid motion
$\boldsymbol q=\boldsymbol K\boldsymbol x$.  Then
\[
0=(\boldsymbol v,\boldsymbol q)_T
=(\div\boldsymbol\tau,\boldsymbol q)_T
=-(\boldsymbol\tau,\boldsymbol K)_T.
\]
Thus $\int_T\skw\boldsymbol\tau\,\dx=0$, and hence
$\boldsymbol\tau\in\mathbb B_k^{\div}(T^{\rm R};\mathbb M)$.
\end{proof}

Applying Proposition~2.3 of
\cite{ChristiansenGopalakrishnanGuzmanHu2024} to
\eqref{eq:localbggbubblediagram_1} gives the exact sequence
\begin{equation}\label{bubble2631}
\begin{aligned}
\begin{bmatrix}
\mathbb B_{k+3}^{\hess}(T^{\rm R};\mathbb R^3)\\
\mathbb B_{k+2}^{\hess}(T^{\rm R};\mathbb R^3)
\end{bmatrix}
&\xrightarrow{[\grad,-\!\mskw]}
\mathbb B_{k+2}^{1,\curl}(T^{\rm R};\mathbb M)
\xrightarrow{\curl S^{-1}\curl}
\mathbb B_k^{\div}(T^{\rm R};\mathbb M)\\
&\xrightarrow{\left[\begin{smallmatrix}2\vskw\\ \div\end{smallmatrix}\right]}
\begin{bmatrix}
\mathbb P_k^{-1}(T^{\rm R};\mathbb R^3)/\mathbb R^3\\
\mathbb P_{k-1}^{-1}(T^{\rm R};\mathbb R^3)/{\rm RM}
\end{bmatrix}
\to0.
\end{aligned}
\end{equation}

\begin{proof}[Proof of Lemma~\ref{lem:bubble-divergence}]
Setting the first component of the final map in \eqref{bubble2631} equal to
zero gives \eqref{divontoB}.  The same exact sequence yields
\[
\dim\mathbb B_k^{\div}(T^{\rm R};\mathbb S)
=\dim\mathbb B_k^{\div}(T^{\rm R};\mathbb M)
-\dim\mathbb P_k^{-1}(T^{\rm R};\mathbb R^3)+3.
\]
Using \eqref{eq:VkHdivdim} gives
$\dim\mathbb B_k^{\div}(T^{\rm R};\mathbb S)
=(k+1)(k+2)(4k-3)$, proving \eqref{eq:bubbledivSdim}.
\end{proof}

The BGG construction also gives the exact complex
\begin{equation}\label{eq:bubbleelascomplex3d_lemma}
\begin{aligned}
\mathbb B_{k+3}^{\hess}(T^{\rm R};\mathbb R^3)
&\xrightarrow{\defm}
\sym(\mathbb B_{k+2}^{1,\curl}(T^{\rm R};\mathbb M))
\xrightarrow{\inc}
\mathbb B_k^{\div}(T^{\rm R};\mathbb S) \xrightarrow{\div}
\mathbb P_{k-1}^{-1}(T^{\rm R};\mathbb R^3)/{\rm RM}
\to0.
\end{aligned}
\end{equation}
Set
\begin{align*}
\mathbb B_{k+2}^{1,\inc}(T^{\rm R};\mathbb S)
:=\{\boldsymbol\tau\in\Sigma_{k+2}^{1,\inc}(T;\mathbb S):
&\ \grad\boldsymbol\tau\text{ vanishes at all vertices of }T,\\
&\ \boldsymbol\tau\text{ and }
(\curl\boldsymbol\tau)^{\intercal}\times\boldsymbol n
\text{ vanish on }\partial T\}.
\end{align*}

\begin{lemma}\label{lem:appendix-intermediate-smooth}
The complex
\begin{equation}\label{eq:bubbleelascomplex03d}
\begin{aligned}
\mathbb B_{k+3}^{\hess}(T^{\rm R};\mathbb R^3)
&\xrightarrow{\defm}
\mathbb B_{k+2}^{1,\inc}(T^{\rm R};\mathbb S)
\xrightarrow{\inc}
\mathbb B_k^{\div}(T^{\rm R};\mathbb S) \xrightarrow{\div}
\mathbb P_{k-1}^{-1}(T^{\rm R};\mathbb R^3)/{\rm RM}
\to0
\end{aligned}
\end{equation}
is exact.
\end{lemma}
\begin{proof}
By \eqref{eq:bubbleelascomplex3d_lemma}, it suffices to prove
\[
\sym(\mathbb B_{k+2}^{1,\curl}(T^{\rm R};\mathbb M))
=\mathbb B_{k+2}^{1,\inc}(T^{\rm R};\mathbb S).
\]
Let $\boldsymbol\tau\in\mathbb B_{k+2}^{1,\curl}(T^{\rm R};\mathbb M)$.
Since $\boldsymbol\tau$ and $\curl\boldsymbol\tau$ vanish on $\partial T$,
\eqref{eq:anticommutprop2.1} and
$2\div(\vskw\boldsymbol\tau)=\tr(\curl\boldsymbol\tau)$ imply
\begin{align*}
(\curl(\sym\boldsymbol\tau))^{\intercal}\times\boldsymbol n
&=-(\curl(\skw\boldsymbol\tau))^{\intercal}\times\boldsymbol n =\big(S(\grad(\vskw\boldsymbol\tau))\big)^{\intercal}
 \times\boldsymbol n\\
&=\curl_F(\vskw\boldsymbol\tau)
-\big(\div(\vskw\boldsymbol\tau)\big)\mskw\boldsymbol n=0
\end{align*}
on $\partial T$.  Thus
$\sym\boldsymbol\tau\in\mathbb B_{k+2}^{1,\inc}(T^{\rm R};\mathbb S)$.

Conversely, let
$\boldsymbol\tau\in\mathbb B_{k+2}^{1,\inc}(T^{\rm R};\mathbb S)$ and set
$\boldsymbol\sigma=\curl S^{-1}\curl\boldsymbol\tau$.
Then $\boldsymbol\sigma\in\mathbb B_k^{\div}(T^{\rm R};\mathbb S)$ and
$\div\boldsymbol\sigma=0$.  By \eqref{bubble2631}, there exists
$\boldsymbol\omega\in\mathbb B_{k+2}^{1,\curl}(T^{\rm R};\mathbb M)$ with
$\boldsymbol\sigma=\curl S^{-1}\curl\boldsymbol\omega$.
Set $\boldsymbol q=S^{-1}\curl(\boldsymbol\tau-\boldsymbol\omega)$.
Then $\boldsymbol q\in\mathbb P_{k+1}^{-1}(T^{\rm R};\mathbb M)
\cap H_0(\curl,T;\mathbb M)$ and $\curl\boldsymbol q=0$.
Hence $\boldsymbol q=\grad\boldsymbol v$ for some
$\boldsymbol v\in\mathbb B_{k+2}^{\grad}(T^{\rm R};\mathbb R^3)$.
For $\boldsymbol\theta=\boldsymbol\tau+\mskw\boldsymbol v$, we have
$\boldsymbol\theta\in
\mathbb B_{k+2}^{\grad}(T^{\rm R};\mathbb M)$ and
\[
\curl\boldsymbol\theta
=\curl\boldsymbol\tau-S\grad\boldsymbol v
=\curl\boldsymbol\omega=0
\qquad\text{on }\partial T.
\]
Thus $\boldsymbol\theta\in
\mathbb B_{k+2}^{1,\curl}(T^{\rm R};\mathbb M)$ and
$\boldsymbol\tau=\sym\boldsymbol\theta$.
\end{proof}

\begin{proof}[Proof of Lemma~\ref{lem:bubble-complex-smooth}]
By Lemma~\ref{lem:appendix-intermediate-smooth} and
$\mathbb B_{k+2}^{1,\inc}(T^{\rm R};\mathbb S)
\subseteq\mathbb B_{k+2}^{\inc}(T^{\rm R};\mathbb S)$, it remains to show
\[
\mathbb B_{k+2}^{\inc}(T^{\rm R};\mathbb S)\cap\ker(\inc)
=\defm(\mathbb B_{k+3}^{\rm herm}(T^{\rm R};\mathbb R^3)).
\]
If $\boldsymbol v\in
\mathbb B_{k+3}^{\rm herm}(T^{\rm R};\mathbb R^3)$, then the boundary
conditions show that $\defm\boldsymbol v$ belongs to the left-hand side.
Conversely, let
$\boldsymbol\tau\in\mathbb B_{k+2}^{\inc}(T^{\rm R};\mathbb S)$ satisfy
$\inc\boldsymbol\tau=0$.  By \eqref{eq:localelascomplex3d}, there exists
$\boldsymbol v\in V_{k+3}^{\hess}(T^{\rm R};\mathbb R^3)$ such that
$\boldsymbol\tau=\defm\boldsymbol v$; fix the rigid motion so that
$\boldsymbol v$ and $\curl\boldsymbol v$ vanish at one vertex.
The identities \eqref{eq:trdef} give
\[
\defm_F(\Pi_F\boldsymbol v)=0,\qquad
\nabla_F^2(\boldsymbol v\cdot\boldsymbol n)=0.
\]
Thus $\Pi_F\boldsymbol v$ is a face rigid motion and
$\boldsymbol v\cdot\boldsymbol n$ is linear on each face.  The normalization
implies $\boldsymbol v|_{\partial T}=0$, so
$\boldsymbol v\in\mathbb B_{k+3}^{\rm herm}(T^{\rm R};\mathbb R^3)$.
\end{proof}

\begin{proof}[Proof of Lemma~\ref{lem:bubble-dimensions-smooth}]
The finite element de Rham complex
\cite{ArnoldFalkWinther2006} and
\cite[Theorem~3.1]{FuGuzmanNeilan2020} give
\begin{equation*}
\begin{aligned}
\mathbb B_{k+3}^{\rm herm}(T^{\rm R})
&\xrightarrow{\grad}
\mathbb P_{k+2}^{\grad}(T^{\rm R};\mathbb R^3)\cap H_0(\curl,T)
\xrightarrow{\curl}
V_{k+1}^{\div}(T^{\rm R})\cap H_0(\div,T)\\
&\xrightarrow{\div}
\mathbb P_k^{-1}(T^{\rm R})/\mathbb R
\to0.
\end{aligned}
\end{equation*}
Therefore
\begin{align*}
\dim\mathbb B_{k+3}^{\rm herm}(T^{\rm R})
&=\dim(\mathbb P_{k+2}^{\grad}(T^{\rm R};\mathbb R^3)
       \cap H_0(\curl,T))\\
&\quad-\dim(V_{k+1}^{\div}(T^{\rm R})\cap H_0(\div,T))
+\dim\mathbb P_k^{-1}(T^{\rm R})-1.
\end{align*}
Using
\begin{align*}
\dim(\mathbb P_{k+2}^{\grad}(T^{\rm R};\mathbb R^3)
       \cap H_0(\curl,T))
&=(k+2)(2k^2+7k+7)+1,\\
\dim(V_{k+1}^{\div}(T^{\rm R})\cap H_0(\div,T))
&=(k+2)(k+3)(2k+3),\\
\dim\mathbb P_k^{-1}(T^{\rm R})
&=\frac{2}{3}(k+1)(k+2)(k+3),
\end{align*}
gives \eqref{eq:bubblehermdim}. Finally,
\eqref{eq:bubbleelascomplex3d},
\eqref{eq:bubblehermdim}, and \eqref{eq:bubbledivSdim} give
\begin{align*}
\dim\mathbb B_{k+2}^{\inc}(T^{\rm R};\mathbb S)
&=\dim\mathbb B_{k+3}^{\rm herm}(T^{\rm R};\mathbb R^3)
+\dim\mathbb B_k^{\div}(T^{\rm R};\mathbb S)\\
&\quad-\dim\mathbb P_{k-1}^{-1}(T^{\rm R};\mathbb R^3)+6
=4k^3+9k^2-k,
\end{align*}
which proves \eqref{eq:BincSdim}.
\end{proof}

\subsection{The less regular bubble complex}

Define
\begin{equation*}
\begin{aligned}
\mathbb B_{k+2}^{\curl,\skw}(T^{\rm R};\mathbb M)
:=\{\boldsymbol\tau\in
&\Sigma_{k+2}^{\curl,\skw}(T;\mathbb M):
\boldsymbol\tau\text{ vanishes at all vertices of }T,\\
&\boldsymbol\tau\times\boldsymbol n,\ \curl\boldsymbol\tau,
\text{ and }\vskw\boldsymbol\tau\text{ vanish on }\partial T\}.
\end{aligned}
\end{equation*}
These spaces form the diagram
\begin{equation}\label{eq:localbggbubblediagram_2}
\begin{tikzcd}[column sep=tiny]
\mathbb B_{k+3}^{1,\curl}(T^{\rm R})
 \arrow{r}{\grad}
&
\mathbb B_{k+2}^{\curl,\skw}(T^{\rm R};\mathbb M)
 \arrow{r}{\curl}
&
\mathbb B_{k+1}^{\grad}(T^{\rm R};\mathbb M)
 \arrow{r}{\div}
&
\mathbb P_k^{-1}(T^{\rm R};\mathbb R^3)/\mathbb R^3
 \to0
\\
\mathbb B_{k+2}^{\hess}(T^{\rm R};\mathbb R^3)
 \arrow[ur,swap,"\mskw"']
 \arrow{r}{\grad}
&
\mathbb B_{k+1}^{\grad}(T^{\rm R};\mathbb M)
 \arrow[ur,swap,"S"']
 \arrow{r}{\curl}
&
\mathbb B_k^{\div}(T^{\rm R};\mathbb M)
 \arrow[ur,swap,"-2\vskw"']
 \arrow{r}{\div}
&
\mathbb P_{k-1}^{-1}(T^{\rm R};\mathbb R^3)/{\rm RM}
 \to0.
\end{tikzcd}
\end{equation}

\begin{lemma}\label{lem:appendix-bubble-derham-less-regular}
The top row of \eqref{eq:localbggbubblediagram_2} is exact.
\end{lemma}
\begin{proof}
By \cite[Theorem~3.1]{FuGuzmanNeilan2020},
\[
\div\mathbb B_{k+1}^{\grad}(T^{\rm R};\mathbb M)
=\mathbb P_k^{-1}(T^{\rm R};\mathbb R^3)/\mathbb R^3.
\]
The same theorem and $\curl\boldsymbol v=2\vskw\grad\boldsymbol v$ give
\[
\grad\mathbb B_{k+3}^{1,\curl}(T^{\rm R})
=\mathbb B_{k+2}^{\curl,\skw}(T^{\rm R};\mathbb M)\cap\ker(\curl).
\]
If $\boldsymbol\sigma\in\mathbb B_{k+1}^{\grad}(T^{\rm R};\mathbb M)$
and $\div\boldsymbol\sigma=0$, then
\cite[Corollary~3.4]{FuGuzmanNeilan2020} gives
$\boldsymbol\tau\in\mathbb B_{k+2}^{\grad}(T^{\rm R};\mathbb M)
\subseteq\mathbb B_{k+2}^{\curl,\skw}(T^{\rm R};\mathbb M)$ such that
$\curl\boldsymbol\tau=\boldsymbol\sigma$.
\end{proof}

Proposition~2.3 of
\cite{ChristiansenGopalakrishnanGuzmanHu2024} gives
\begin{equation}\label{bubble263}
\begin{aligned}
\begin{bmatrix}
\mathbb B_{k+3}^{1,\curl}(T^{\rm R})\\
\mathbb B_{k+2}^{\hess}(T^{\rm R};\mathbb R^3)
\end{bmatrix}
&\xrightarrow{[\grad,-\!\mskw]}
\mathbb B_{k+2}^{\curl,\skw}(T^{\rm R};\mathbb M)
\xrightarrow{\curl S^{-1}\curl}
\mathbb B_k^{\div}(T^{\rm R};\mathbb M)\\
&\xrightarrow{\left[\begin{smallmatrix}2\vskw\\ \div\end{smallmatrix}\right]}
\begin{bmatrix}
\mathbb P_k^{-1}(T^{\rm R};\mathbb R^3)/\mathbb R^3\\
\mathbb P_{k-1}^{-1}(T^{\rm R};\mathbb R^3)/{\rm RM}
\end{bmatrix}
\to0.
\end{aligned}
\end{equation}
Consequently,
\begin{equation}\label{eq:bubbleelascomplex13d_lemma}
\begin{aligned}
\mathbb B_{k+3}^{1,\curl}(T^{\rm R})
&\xrightarrow{\defm}
\sym(\mathbb B_{k+2}^{\curl,\skw}(T^{\rm R};\mathbb M))
\xrightarrow{\inc}
\mathbb B_k^{\div}(T^{\rm R};\mathbb S) \xrightarrow{\div}
\mathbb P_{k-1}^{-1}(T^{\rm R};\mathbb R^3)/{\rm RM}
\to0
\end{aligned}
\end{equation}
is exact.

\begin{proof}[Proof of Lemma~\ref{lem:bubble-complex-less-regular}]
By \eqref{eq:bubbleelascomplex13d_lemma}, it suffices to prove
\[
\sym(\mathbb B_{k+2}^{\curl,\skw}(T^{\rm R};\mathbb M))
=\mathbb B_{k+2}^{\inc^+}(T^{\rm R};\mathbb S).
\]
Let $\boldsymbol\tau\in
\mathbb B_{k+2}^{\curl,\skw}(T^{\rm R};\mathbb M)$.
Then $(\sym\boldsymbol\tau)\times\boldsymbol n=0$ on $\partial T$.
Moreover, \eqref{eq:anticommutprop2.1} and
$2\div(\vskw\boldsymbol\tau)=\tr(\curl\boldsymbol\tau)$ give
\begin{align*}
(\curl(\sym\boldsymbol\tau))^{\intercal}\times\boldsymbol n
&=-(\curl(\skw\boldsymbol\tau))^{\intercal}\times\boldsymbol n =\big(S(\grad(\vskw\boldsymbol\tau))\big)^{\intercal}
 \times\boldsymbol n\\
&=\curl_F(\vskw\boldsymbol\tau)
-\big(\div(\vskw\boldsymbol\tau)\big)\mskw\boldsymbol n=0.
\end{align*}
Thus $\sym\boldsymbol\tau\in
\mathbb B_{k+2}^{\inc^+}(T^{\rm R};\mathbb S)$.

Conversely, let
$\boldsymbol\tau\in\mathbb B_{k+2}^{\inc^+}(T^{\rm R};\mathbb S)$ and set
$\boldsymbol\sigma=\curl S^{-1}\curl\boldsymbol\tau$.
Then $\boldsymbol\sigma\in
\mathbb B_k^{\div}(T^{\rm R};\mathbb S)$ and
$\div\boldsymbol\sigma=0$.
By \eqref{bubble263}, choose
$\boldsymbol\omega\in
\mathbb B_{k+2}^{\curl,\skw}(T^{\rm R};\mathbb M)$ such that
$\boldsymbol\sigma=\curl S^{-1}\curl\boldsymbol\omega$.
Then
$\boldsymbol q=S^{-1}\curl(\boldsymbol\tau-\boldsymbol\omega)$ belongs to
$\mathbb P_{k+1}^{-1}(T^{\rm R};\mathbb M)\cap
H_0(\curl,T;\mathbb M)$ and satisfies $\curl\boldsymbol q=0$.
Write $\boldsymbol q=\grad\boldsymbol v$ with
$\boldsymbol v\in\mathbb B_{k+2}^{\grad}(T^{\rm R};\mathbb R^3)$.
For $\boldsymbol\theta=\boldsymbol\tau+\mskw\boldsymbol v$, we have
$\boldsymbol\theta\in
\mathbb P_{k+1}^{-1}(T^{\rm R};\mathbb M)
\cap H_0(\curl,T;\mathbb M)$, and it vanishes at all vertices of $T$.
Moreover,
\[
\curl\boldsymbol\theta
=\curl\boldsymbol\tau-S\grad\boldsymbol v
=\curl\boldsymbol\omega=0
\qquad\text{on }\partial T.
\]
Thus $\boldsymbol\theta\in
\mathbb B_{k+2}^{\curl,\skw}(T^{\rm R};\mathbb M)$ and
$\boldsymbol\tau=\sym\boldsymbol\theta$.
\end{proof}

\begin{proof}[Proof of Lemma~\ref{lem:bubble-dimensions-less-regular}]
By \cite[p.~1076]{FuGuzmanNeilan2020},
\[
\dim\mathbb B_{k+3}^{1,\curl}(T^{\rm R})=k(k+1)(2k+3),
\]
which proves \eqref{eq:bubble1curldim}.  The exactness of
\eqref{eq:bubbleelascomplex13d} and \eqref{eq:bubbledivSdim} gives
\begin{align*}
\dim\mathbb B_{k+2}^{\inc^+}(T^{\rm R};\mathbb S)
&=\dim\mathbb B_{k+3}^{1,\curl}(T^{\rm R})
+\dim\mathbb B_k^{\div}(T^{\rm R};\mathbb S)\\
&\quad-\dim\mathbb P_{k-1}^{-1}(T^{\rm R};\mathbb R^3)+6
=4k^3+8k^2-2k,
\end{align*}
which proves \eqref{eq:Binc+Sk2}.
\end{proof}

\section*{Acknowledgments}
The authors thank Professor Long Chen (University of California, Irvine) for helpful discussions at the Workshop on Finite Element Tensor Calculus held at the Tsinghua Sanya International Mathematics Forum.  The first author thanks Professor Jeonghun Lee (Baylor University) for several insightful conversations.  Part of this work was also discussed during the thematic programme ``Differential Complexes: Theory, Discretization, and Applications'' at the Erwin Schr\"odinger International Institute for Mathematics and Physics.  The authors gratefully acknowledge the hospitality of both institutes.

\bibliographystyle{abbrv}
\bibliography{./references}

\begin{thebibliography}{10}

\bibitem{amstutz2018incompatibility}
S.~Amstutz and N.~Van~Goethem.
\newblock {The incompatibility operator: from Riemann's intrinsic view of geometry to a new model of elasto-plasticity}.
\newblock In {\em Topics in Applied Analysis and Optimisation}, pages 33--70. Springer, 2019.

\bibitem{ArnoldAwanouWinther2008}
D.~Arnold, G.~Awanou, and R.~Winther.
\newblock Finite elements for symmetric tensors in three dimensions.
\newblock {\em Math. Comp.}, 77(263):1229--1251, 2008.

\bibitem{ArnoldDouglasGupta1984}
D.~N. Arnold, J.~Douglas, Jr., and C.~P. Gupta.
\newblock A family of higher order mixed finite element methods for plane elasticity.
\newblock {\em Numer. Math.}, 45(1):1--22, 1984.

\bibitem{Arnold2006a}
D.~N. Arnold, R.~S. Falk, and R.~Winther.
\newblock Differential complexes and stability of finite element methods {II}: The elasticity complex.
\newblock In D.~Arnold, P.~Bochev, R.~Lehoucq, R.~Nicolaides, and M.~Shashkov, editors, {\em Compatible Spatial Discretizations}, volume 142 of {\em IMA Vol. Math. Appl.}, pages 47--68. Springer, Berlin, 2006.

\bibitem{ArnoldFalkWinther2006}
D.~N. Arnold, R.~S. Falk, and R.~Winther.
\newblock Finite element exterior calculus, homological techniques, and applications.
\newblock {\em Acta Numer.}, 15:1--155, 2006.

\bibitem{ArnoldHu2021}
D.~N. Arnold and K.~Hu.
\newblock Complexes from complexes.
\newblock {\em Found. Comput. Math.}, 21(6):1739--1774, 2021.

\bibitem{ArnoldWinther2002}
D.~N. Arnold and R.~Winther.
\newblock Mixed finite elements for elasticity.
\newblock {\em Numer. Math.}, 92(3):401--419, 2002.

\bibitem{BrezziDouglasMarini1985}
F.~Brezzi, J.~Douglas, Jr., and L.~D. Marini.
\newblock Two families of mixed finite elements for second order elliptic problems.
\newblock {\em Numer. Math.}, 47(2):217--235, 1985.

\bibitem{ChenHuang2022c}
L.~Chen and X.~Huang.
\newblock Discrete {H}essian complexes in three dimensions.
\newblock In {\em The virtual element method and its applications}, volume~31 of {\em SEMA SIMAI Springer Ser.}, pages 93--135. Springer, Cham, 2022.

\bibitem{ChenHuang2022}
L.~Chen and X.~Huang.
\newblock A finite element elasticity complex in three dimensions.
\newblock {\em Math. Comp.}, 91(337):2095--2127, 2022.

\bibitem{ChenHuang2022b}
L.~Chen and X.~Huang.
\newblock Finite elements for div- and divdiv-conforming symmetric tensors in arbitrary dimension.
\newblock {\em SIAM J. Numer. Anal.}, 60(4):1932--1961, 2022.

\bibitem{ChenHuang2022a}
L.~Chen and X.~Huang.
\newblock Finite elements for {${\rm div\,div}$} conforming symmetric tensors in three dimensions.
\newblock {\em Math. Comp.}, 91(335):1107--1142, 2022.

\bibitem{ChenHuang2025a}
L.~Chen and X.~Huang.
\newblock Finite element complexes in two dimensions.
\newblock {\em Sci. Sin. Math.}, 55(8):1593--1626, 2025.

\bibitem{ChenHuang2025}
L.~Chen and X.~Huang.
\newblock Hybridizable symmetric stress elements on the barycentric refinement in arbitrary dimensions.
\newblock {\em Math. Comp.}, https://arxiv.org/abs/2501.02691, 2025.

\bibitem{ChenHuang2026}
L.~Chen and X.~Huang.
\newblock Complexes from complexes: finite element complexes in three dimensions.
\newblock {\em Math. Comp.}, 95(359):1083--1142, 2026.

\bibitem{ChristiansenGopalakrishnanGuzmanHu2024}
S.~H. Christiansen, J.~Gopalakrishnan, J.~Guzm\'an, and K.~Hu.
\newblock A discrete elasticity complex on three-dimensional {A}lfeld splits.
\newblock {\em Numer. Math.}, 156(1):159--204, 2024.

\bibitem{ChristiansenHu2023}
S.~H. Christiansen and K.~Hu.
\newblock Finite element systems for vector bundles: elasticity and curvature.
\newblock {\em Found. Comput. Math.}, 23(2):545--596, 2023.

\bibitem{ciarlet2009intrinsic}
P.~G. Ciarlet, L.~Gratie, and C.~Mardare.
\newblock Intrinsic methods in elasticity: A mathematical survey.
\newblock {\em Discrete and Continuous Dynamical Systems}, 2009.

\bibitem{FuGuzmanNeilan2020}
G.~Fu, J.~Guzm\'{a}n, and M.~Neilan.
\newblock Exact smooth piecewise polynomial sequences on {A}lfeld splits.
\newblock {\em Math. Comp.}, 89(323):1059--1091, 2020.

\bibitem{geymonat2005some}
G.~Geymonat and F.~Krasucki.
\newblock Some remarks on the compatibility conditions in elasticity.
\newblock {\em Accad. Naz. Sci. XL}, 123:175--182, 2005.

\bibitem{GongGopalakrishnanGuzmanNeilan2023}
S.~Gong, J.~Gopalakrishnan, J.~Guzm\'{a}n, and M.~Neilan.
\newblock Discrete elasticity exact sequences on {W}orsey-{F}arin splits.
\newblock {\em ESAIM Math. Model. Numer. Anal.}, 57(6):3373--3402, 2023.

\bibitem{GopalakrishnanGuzmanLee2025}
J.~Gopalakrishnan, J.~Guzmán, and J.~J. Lee.
\newblock The {Johnson–K\v{r}\'{i}\v{z}ek–Mercier} elasticity element in higher dimensions.
\newblock {\em J. Numer. Math.}, 2025.

\bibitem{Hu2015}
J.~Hu.
\newblock Finite element approximations of symmetric tensors on simplicial grids in {$\Bbb R^n$}: the higher order case.
\newblock {\em J. Comput. Math.}, 33(3):283--296, 2015.

\bibitem{HuMaZhang2021}
J.~Hu, R.~Ma, and M.~Zhang.
\newblock A family of mixed finite elements for the biharmonic equations on triangular and tetrahedral grids.
\newblock {\em Sci. China Math.}, 64(12):2793--2816, 2021.

\bibitem{HuZhang2015}
J.~Hu and S.~Zhang.
\newblock A family of symmetric mixed finite elements for linear elasticity on tetrahedral grids.
\newblock {\em Sci. China Math.}, 58(2):297--307, 2015.

\bibitem{HuangZhangZhouZhu2024}
X.~Huang, C.~Zhang, Y.~Zhou, and Y.~Zhu.
\newblock New low-order mixed finite element methods for linear elasticity.
\newblock {\em Adv. Comput. Math.}, 50(2):Paper No. 17, 31, 2024.

\bibitem{JohnsonMercier1978}
C.~Johnson and B.~Mercier.
\newblock Some equilibrium finite element methods for two-dimensional elasticity problems.
\newblock {\em Numer. Math.}, 30(1):103--116, 1978.

\bibitem{Krizek1982}
M.~K\v{r}\'{i}\v{z}ek.
\newblock An equilibrium finite element method in three-dimensional elasticity.
\newblock {\em Apl. Mat.}, 27(1):46--75, 1982.

\bibitem{seeger1961recent}
A.~Seeger.
\newblock Recent advances in the theory of defects in crystals.
\newblock {\em Physica Status Solidi (B)}, 1(7):669--698, 1961.

\end{thebibliography}

\end{document}